\documentclass[11pt]{amsart}

\usepackage{amsmath}
\usepackage{amssymb}
\usepackage{amsthm}
\usepackage[cp1250]{inputenc}

\usepackage[T1]{fontenc}
\usepackage[english]{babel}
\usepackage{graphicx}
\usepackage{amsfonts}
\usepackage{color}

\usepackage{mathtools}

\mathtoolsset{showonlyrefs}

\theoremstyle{plain}
 \newtheorem{thm}{Theorem}[section]
 \newtheorem{cor}[thm]{Corollary}
 \newtheorem{lem}[thm]{Lemma}
 \newtheorem{prop}[thm]{Proposition}
 \theoremstyle{definition}
 
 \theoremstyle{remark}
 \newtheorem{rem}[thm]{Remark}
 
 \numberwithin{equation}{section}

\renewcommand\Re{\operatorname{Re}}
\renewcommand\Im{\operatorname{Im}}

\begin{document}

\title[Asymptotic Behavior of Solutions to the Dirac System]{
Asymptotic Behavior of Solutions \\to the Dirac System
with Respect \\
to a Spectral Parameter
 }

\author{Alexander Gomilko}
\address{Faculty of Mathematics and Computer Science\\
Nicolas Copernicus University\\
Chopin Street 12/18\\
87-100 Toru\'n, Poland}
\email{alex@gomilko.com}

\author{{\L}ukasz Rzepnicki}
\address{Faculty of Mathematics and Computer Science\\
Nicolas Copernicus University\\
Chopin Street 12/18\\
87-100 Toru\'n, Poland}
\email{keleb@mat.umk.pl}


\def\today{\number\day \space\ifcase\month\or
 January\or February\or March\or April\or May\or June\or
 July\or August\or September\or October\or November\or December\fi
 \space \number\year}

\thanks{This work was supported by the NAWA/NSF grant BPN/NSF/2023/1/00001. It was also partially supported by the NCN grant UMO-2023/49/B/ST1/01961}


\begin{abstract}
We consider the Dirac system of ordinary differential equations
\[
Y'(x) +
\begin{bmatrix}
0 & \sigma_1(x) \\
\sigma_2(x) & 0
\end{bmatrix}
Y(x) = i\mu
\begin{bmatrix}
1 & 0 \\
0 & -1
\end{bmatrix}
Y(x), \quad
Y(x) =
\begin{bmatrix}
y_1(x) \\
y_2(x)
\end{bmatrix},
\]
where \( x \in [0,1] \), \( \mu \in \mathbb{C} \) is a spectral parameter, and \( \sigma_j \in L^p[0,1] \), \( j = 1,2 \), for \( p \in [1,2] \). We study the asymptotic behavior of the system's fundamental solutions as \( |\mu| \to \infty \) in the half-plane \( \operatorname{Im} \mu > -r \), where \( r \geq 0 \) is fixed, and obtain detailed asymptotic formulas. As an application, we derive new results on the half-plane asymptotics of fundamental system of solutions to Sturm--Liouville equations with singular potentials.
\end{abstract}

\maketitle

\section{Introduction}

The main objective of this paper is to study the asymptotic behavior of solutions to the Dirac system with an integrable potential, with respect to a spectral parameter located in the complex half-plane. This article addresses two main topics. The first concerns the asymptotic behavior of a certain fundamental set of solutions associated with a Dirac system and to the Cauchy problem for Dirac systems. The second focuses on results related to Sturm--Liouville equations with singular potentials.
The approach presented in this paper is based on transforming the system into appropriate integral equations
as developed in \cite{GPA}, \cite{SaKos} and \cite{Rzep25}.

In recent years, Dirac operators have attracted a considerable attention, particularly in the field of non-self-adjoint spectral theory. From the perspective of physics, they are relevant in the study of nonlinear Schr\"{o}dinger equations and have also been employed as effective models for graphene (see \cite{Siegl} or \cite{Korot14} and references therein).

There is a wide variety of asymptotic results in the literature concerning solutions to Dirac systems, obtained in different settings and under a wide range of assumptions on the potentials. Most of these results assume that the spectral parameter lies within a strip in the complex plane, and they typically provide only the leading asymptotic term. We introduce a new and relatively simple method to derive sharp asymptotics
for fundamental system of solutions to Dirac systems
with integrable coefficients which is analytic with respect to spectral parameter taking values in a half-plane. Moreover, our results can be used to cover spectral parameters in the whole plane for large spectral parameter.

To the best of our knowledge,
so precise results in this setting
have not previously appeared in the literature, apart from the recent articles \cite{SaKos21} and \cite{SaKos} focused
on the systems with smooth coefficients.

For $Y=[y_1,y_2]^T$ we consider an one-dimensional  Dirac system in the following form:
\begin{equation}\label{v120}
Y'(x)+J(x) Y(x)=i\mu A Y(x),
\end{equation}
for $x\in [0,1],$ with matrices
\begin{equation}\label{matrixA}
J(x)=\left[\begin{array}{cc}
0 & \sigma_1(x)\\
\sigma_2(x) & 0
\end{array}
\right], \quad A=\left[\begin{array}{cc}
1 & 0\\
0 &-1
\end{array}
\right].
\end{equation}
Here $\mu\in \mathbb{C}$ is a spectral parameter, and  complex-valued functions $\sigma_1$ and $\sigma_2$ belong to
$L^p[0,1]$, with $1\leq p \leq 2$. We shall restrict our attention to such an interval for
$p$, since the case $p=1$ is the most general one, whereas $p=2$ is required in view of applications
to Sturm--Liouville equations with singular potentials.

The solution $Y=[y_1, y_2]^T$ of \eqref{v120} is understood to be a function having entries $y_1$ and $y_2$ from
the space of absolutely continuous  functions on $[0,1]$
(i.e. from the Sobolev space $W^{1,1}[0,1]$)
and satisfying  \eqref{v120} for a.e. $x\in [0,1]$. All our results are formulated under the assumption that the spectral parameter is large enough and lies in the half-plane
\begin{align}
\Sigma_r=\{\mu \in \mathbb{C} : \; \Im \mu > -r \}, \ \ r\geq0. \label{half}
\end{align}

Results involving or relying on asymptotics of fundamental set of solutions for Dirac systems are widespread in the literature. In particular, a number of classical results related to the subject of the paper can be found in \cite{BirLan}, \cite{Levi}, \cite{Rykhlov99} and \cite{March}. For a comprehensive bibliography along with a historical overview see
\cite{SSasy} and \cite{SaKos}.

Note that fine asymptotics of fundamental set of solutions
is crucial in the study of
 basis property for the system of root vectors of the Dirac operator with different types of boundary conditions.
Such a property for the Dirac system \eqref{v120} with strictly regular boundary conditions and a summable potential matrix, was independently established by A. A. Lunyov and
M. M. Malamud in \cite{LM14}, \cite{LMal}, and by A. M. Savchuk and A. A. Shkalikov in \cite{SSDirac}.
The proofs in these papers relied on sharp asymptotic formulas for solutions of \eqref{v120}, assuming that the spectral parameter lies within a horizontal strip in the complex plane.
Among other relevant works we mention \cite{MalOri12}, \cite{DM},
\cite{Makin20} and \cite{Makin24}.

Various methods for analyzing solutions of Dirac systems are available in the literature. We note, in particular, a  method, employing the Pr\"{u}fer's transformation, is discussed in \cite{S} and \cite{SSDirac}. Another line of research, relying on  the technique of transformation operators, is developed in \cite{LM20}, \cite{L23}, and \cite{LMal24}. Although,  to obtain finer results, other approaches were used as well.

It is often of value to have a method providing sharp asymptotics results without invoking very involved and demanding techniques.
On this way, very detailed asymptotic formulas were established in \cite{GRZ1} and \cite{Rzep21} using a relatively simple idea. In \cite{GRZ1}, sharp asymptotic formulas for the fundamental set of solutions
of \eqref{v120} were obtained under the assumption that the potential $J$ is given by $\sigma_1, \sigma_2 \in L^2[0,1].$ They found applications to study of the Sturm-Liouville problem involving a singular potential. A recent work \cite{Rzep21} extended the analysis in \cite{GRZ1} to the case when
$\sigma_1, \sigma_2 \in L^p[0,1]$ with $1 \leq p <  2$. In both papers, the spectral parameter was considered within a horizontal strip
and the results were based on a specific integral representation of the fundamental matrix $D$ of \eqref{v120}
satisfying $D(0)=I$, where $I$ is the identical matrix.

To study the problem with the spectral parameter in a half-plane, one needs to employ different methods. 
It relies on obtaining a suitable representation of solutions to Dirac system using integral operators as in \cite{Rzep25}. The fundamental system of the solutions to the Dirac system is expressed with the use of a power series of these operators. Given the desired behavior of the reminder,  
one obtains a bound for the series of higher-order terms allowing to reduce the study of asymptotics
of fundamental system
to careful analysis of a few lower order terms.
The obtained results are applied to the study of the Cauchy problem for the Dirac system in the half-plane.
Having established asymptotic formulas when the spectral parameter belongs to a half-plane, it is easy to derive results
 in the entire complex plane, except for a bounded neighborhood around zero (see Remark \ref{remNeg}).

The paper is organized as follows.
First, we recall the results from \cite{Rzep25} concerning a fundamental system of solutions to Dirac system. Next we
improve some estimates of integral operators and consequently for solutions to Dirac systems.  Then we establish  asymptotic formulas for a fundamental system of solutions
to \eqref{v120}. The obtained asymptotic formulas are then applied to the study of the Cauchy problem with spectral parameter in a half-plane.
We also compare our results with asymptotic formulas from  \cite{SaKos21}, \cite{SaKos} and those proven in \cite{Rzep21} (where spectral parameter belongs to a strip).
As another application,  we obtain new, sharp
asymptotic formulas for fundamental system of solutions
of  Sturm-Liouville equations with singular potentials.
These findings demonstrate a substantially higher accuracy compared to the known results obtained in
the strip \cite{SSsp} and in the half-plane \cite{Vladyk} and  \cite{SSad3}.

In what follows, we will use the notation $\|\cdot\|_\mathcal{C}$ for the norm in the space of continuous functions  $\mathcal C[0,1]$
and similar notation for $L^p$ spaces.

\section{Integral equations}
The considerations in this paper are substantially related to the recent paper
\cite{Rzep25}, so we need to recall its certain important points.
The paper
concentrated on the study of asymptotics  of  the fundamental system for  \eqref{v120}
given by the functions $W=[w_1,w_2]^T$ and $V=[v_1,v_2]^T$ satisfying
the following boundary conditions
\begin{equation}\label{BC}
w_1(0)=1,\;w_2(1)=0,\quad v_1(0)=0,\;v_2(1)=1.
\end{equation}

If $Y=[y_1, y_2]^T$, the explicit form of the system \eqref{v120} is
\begin{align}
y_1'+\sigma_1y_2&=i\mu y_1, \nonumber\\
y_2'+\sigma_2y_1&=-i\mu y_2. \label{sys}
\end{align}
Next, considering $\mu \in \mathbb C$ with $\Im \mu > -r$, $r \geq 0$, we apply general formulas for the
solutions of first order ordinary differential equations to derive
\begin{align}
y_1(x)&=C_1e^{i \mu x} - \int_0^x e^{i \mu (x-t)}\sigma_1(t)y_2(t)dt,\label{sys2}\\
y_2(x)&=C_2e^{i \mu(1- x)} + \int_x^1 e^{-i \mu (x-t)}\sigma_2(t)y_1(t)dt,\nonumber
\end{align}
with arbitrary constants $C_1$ and $C_2$.

Let $W=[w_1,w_2]^T$ be the solution of \eqref{sys2} with
$C_1=1$ and $C_2=0$. Then $W$ can be expressed in terms of $z_1,$ a solution of specific integral equation. Namely,
for $x\in [0,1],$
\begin{align}
w_1(x)&=e^{i\mu x}z_1(x),\\
w_2(x)&=\int_x^1 e^{-i \mu (x-t)}\sigma_2(t)w_1(t)dt,\label{w2}
\end{align}
where $z_1 \in \mathcal{C}[0,1]$ is a solution of the integral equation
\begin{align}\label{K1E}
z_1+K_1z_1=e,
\end{align}
where
\begin{align}
e(x)=1 \ \ x\in [0,1],
\end{align}
and the integral operator $K_1=K_1(\mu)$ is given by
\begin{align}\label{Op1}
(K_1z)(x)=\int_0^x \sigma_1(t)\int_t^1e^{2i\mu(s-t)}\sigma_2(s)z(s)dsdt, \ \ x\in[0,1].
\end{align}

If $C_1=0$ and $C_2=1,$ then we denote the solution of \eqref{sys2} as $V=[v_1,v_2]^T$.
Similarly to the  case above we establish
\begin{align}
v_1(x)&= - \int_0^x e^{i \mu (x-t)}\sigma_1(t)v_2(t)dt,\label{v1}\\
v_2(x)&=e^{i\mu(1-x)}z_2(x),
\end{align}
where $z_2\in \mathcal{C}[0,1]$ is the solution of
\begin{align}\label{K2E}
z_2+K_2z_2=e,
\end{align}
and the operator $K_2=K_2(\mu)$ is defined by
\begin{align}\label{Op2}
(K_2z)(x)=\int_x^1 \sigma_2(t)\int_0^te^{2i\mu(t-s)}\sigma_1(s)z(s)dsdt, \qquad x\in[0,1].
\end{align}

It is worth to note that operators $K_j$ are connected via simple transformation
\begin{align}\label{simT}
\mathcal{T}x=1-x,\quad x\in [0,1].
\end{align}
Indeed, the operator $K_2$ can be expressed as
\begin{equation}\label{Con}
(K_{2}z)(x)
=(\hat{K}_{1}\hat{z})(\mathcal{T}x),\ \ z\in\mathcal{C}[0,1], \ \  x\in [0,1],
\end{equation}
with the integral operator
\[
(\hat{K}_1z)(x)=\int_0^x \hat{\sigma}_2(t)
\int_t^1 e^{2i\mu(s-t)}\hat{\sigma}_1(s)z(s)\,ds\,dt.
\]
and functions
\[
\hat{\sigma}_j(x):=\sigma_j(\mathcal{T}x),\quad j=1,2, \quad\hat{z}(x):=z(\mathcal{T}x).
\]
This fact allows to focus on estimates of the bounded linear operator $K_1$.

Note that in fact operators $K_j=K_j(\mu)$, $j=1,2$ defined on $\mathcal{C}[0,1]$ are analytic with respect to $\mu \in \mathbb{C}$.
We also see that  results on asymptotics of solutions to the integral equations
\eqref{K1E} and \eqref{K2E} have immediate consequences for
asymptotic analysis of
certain fundamental system for \eqref{v120}.
In turn, the study of \eqref{K1E} and \eqref{K2E} requires obtaining sharp bounds for the
powers of operators $K_j$'s. To this aim we will need
functions
\begin{equation}\label{Dgamma0}
\gamma_{0,1}(x,\mu):=\left|\int_0^x e^{2i\mu(x-t)}\sigma_1(t)\,dt\right|,
\quad \gamma_{0,2}(x,\mu):=\left|\int_x^1 e^{2i\mu(t-x)}\sigma_2(t)\,dt\right|.
\end{equation}
and letting $\sigma_j\in L^p[0,1]$, $j=1,2$ and assuming that
$1/p+1/q=1$, $q\in [2,\infty],$
we  define for $p>1$
\begin{align}\label{Dgamma}
\gamma_{1,q}(\mu)&:=\|\gamma_{0,1}(\cdot,\mu)\|_{L_q}=
\left(\int_0^1 \left|\int_0^s e^{2i\mu(s-t)}
\sigma_1(t)\,dt\right|^q \,ds\right)^{1/q},\\
\gamma_{2,q}(\mu)&:=\|\gamma_{0,2}(\cdot,\mu)\|_{L_q}=
\left(\int_0^1\left|\int_t^1 e^{2i\mu(s-t)}
\sigma_2(s)\,ds\right|^q\,dt\right)^{1/q},\nonumber
\end{align}
and for $p=1$
\begin{equation}\label{DgammaInf}
\gamma_{j,\infty}(\mu)=\sup_{x\in [0,1]}\gamma_{0,j}(x,\mu), \ \  j=1,2.
\end{equation}
We also introduce the notation
\[
v(r):=e^{2r},\quad r\ge 0
\]
and
\begin{equation}\label{sigma}
\sigma_0(x):=\max\{|\sigma_1(x)|, |\sigma_2(x)|\}, \qquad x \in [0,1].
\end{equation}
Note that
\[\sigma_0 \in L^p[0,1], \qquad \text{and}\qquad
\|\sigma_j\|_{L^1}\le \|\sigma_j\|_{L^p}\le \|\sigma_0\|_{L^p}
\]
for $j=1,2$.

It is instructive to recall that
\[
\|\gamma_{0,j}(\cdot,\mu)\|_{\mathcal C}\to 0,\quad \gamma_{j,q}(\mu)\to 0 \qquad \mu\in \Sigma_r,\quad |\mu|\to\infty, \qquad j=1,2,
\]
and using appropriate estimates, one may establish that the norms of operators $\|K_j\|_\mathcal{C}$ go to zero as
$\mu \in \Sigma_r$ and $|\mu| \to \infty$
(see Corollary 3.2 and Lemma 4.2 in \cite{Rzep25}).
Furthermore, recall that by \cite[Lemma 4.1]{Rzep25} for $\mu \in \Sigma_r$ we have
\begin{equation}\label{Ae}
\|K_1e\|_\mathcal{C}\le \|\sigma_1\|_{L^p}\gamma_{2,q}(\mu), \ \ \|K_2e\|_\mathcal{C}\le \|\sigma_2\|_{L^p}\gamma_{1,q}(\mu).
\end{equation}

The next bound for powers of $K_j$'s was obtained in \cite{Rzep25}.
\begin{lem}\label{K1M}
If $\mu\in \Sigma_r$ and $z\in \mathcal{C} [0,1]$, then for every $n \ge 0,$
\begin{align}\label{K120}
\|K_j^{2n}z\|_{\mathcal C}\le  a^n \gamma_{j,q}^n(\mu)\|z\|_{\mathcal C},
\end{align}
where $a:=2v^2(r)\|\sigma_0\|^3_{L^p}$ and $\sigma_0$ is given by \eqref{sigma}.
\end{lem}

Let us finally recall the main results of  \cite{Rzep25},
which are summarized in the following theorem.
For $r\geq0,$ we define
\begin{align}
\Sigma_r(R):=\{\mu \in \Sigma_r \; \ |\mu| > R \}, \ \ R\geq 0, \label{halfR}
\end{align}
where $\Sigma_r$ is given by \eqref{half}.
\begin{thm}\label{AsG1}
Let $j \in \{1,2\}$ and $r\ge 0$ be fixed, and $R_r>0$ is so large that
\begin{equation}
\label{AA}
\gamma_{j,q}(\mu)\le 1/(2a)\quad \mbox{for}\quad \mu\in \Sigma_r(R_r).
\end{equation}
Then, for every $\mu\in \Sigma_r(R_r)$
the equation
\[
z_j(x)+(K_jz)(x)=e(x), \qquad x \in [0,1],
\]
has the unique solution  $z_j\in \mathcal{C}[0,1]$ given by
\begin{equation}\label{Aeq}
z_j=\sum_{n=0}^\infty (-1)^n K^n_je,
\end{equation}
which is analytic with respect to $\mu \in \Sigma_r(R_r)$ and
satisfies
\begin{equation}\label{Ase}
\left\|z_j-e+K_je-K_j^2e\right\|_\mathcal{C}
\le c \Big(\gamma^2_{1,q}(\mu)+\gamma^2_{2,q}(\mu)\Big).
\end{equation}
for some $c>0.$
\end{thm}

Our first objective is to refine the estimate \eqref{Ase} obtained in the preceding theorem.
To this end, it is essential to derive an appropriate bound for $K^2_je,$.
This, in turn, will enable us to control the higher-order terms of the series representing the solutions
$z_j$, by means of a new type of remainder given by
\begin{align}\label{gammaTilde}
\widetilde{\gamma}(\mu):=\int_0^1|\sigma_2(s)|\gamma^2_{0,1}(s,\mu)ds+\int_0^1|\sigma_1(s)|\gamma^2_{0,2}(s,\mu)ds.
\end{align}
Note that
\begin{align*}
\widetilde{\gamma}(\mu) \leq v(r) \|\sigma_0\|^2_{L^p} \Big(\gamma_{1,q}(\mu) + \gamma_{2,q}(\mu)\Big),
\end{align*}
where $\sigma_0$ is given by \eqref{sigma}. Indeed,
$$\|\gamma_{0,j}(\cdot,\mu)\|_{\mathcal C}\leq v(r)\|\sigma_j\|_{L^1}$$
and
\begin{align*}
\widetilde{\gamma}(\mu) \leq v(r) \Big(\|\sigma_1\|_{L^1}\int_0^1 |\sigma_2(s)|\gamma_{0,1}(s,\mu)ds
+\|\sigma_2\|_{L^1} \int_0^1 |\sigma_1(s)|\gamma_{0,2}(s,\mu)ds\Big).
\end{align*}
Hence by the H\"{o}lder inequality we get the desired estimate.

The crucial estimate of $K^2_je,$ depends on the next auxiliary lemma.
\begin{lem}\label{estKint}
Let $r\ge 0$ and $\mu \in \Sigma_r$, $x\in [0,1].$  Then
\begin{align}
\Big|\int_x^1 e^{2i \mu t }\sigma_2(t)(K_1e)(t)dt\Big|&\leq v(r) \Big( \widetilde{\gamma}(\mu))+\|\sigma_1\|_{L^p}\gamma_{2,q}(\mu)\gamma_{0,2}(x,u)\Big),\label{intK1}\\
\Big|\int_0^x e^{2i \mu(x- t) }\sigma_1(t)(K_2e)(t)dt\Big|&\leq v(r) \Big( \widetilde{\gamma}(\mu))+\|\sigma_2\|_{L^p}\gamma_{1,q}(\mu)\gamma_{0,1}(x,u)\Big).\nonumber
\end{align}
Moreover, for $\mu \in \Sigma_r$ we also have
\begin{align*}
\Big|\int_0^1 e^{2i \mu t }\sigma_2(t)(K_1e)(t)dt\Big|&\leq v(r) \widetilde{\gamma}(\mu),\\
\Big|e^{2i \mu}\int_0^1 e^{-2i \mu t }\sigma_1(t)(K_2e)(t)dt\Big|&\leq v(r)  \widetilde{\gamma}(\mu).
\end{align*}
\end{lem}

\begin{proof}
Changing the order of integration in the left hand side of \eqref{intK1}, we obtain
\begin{align}
\int_x^1& e^{2i \mu t }\sigma_2(t)(K_1e)(t)dt\nonumber\\
&=\int_x^1 e^{2i \mu t }\sigma_2(t)\int_0^t\sigma_1(s)\int_s^1 e^{2i \mu (\tau-s) }\sigma_2(\tau)d \tau ds dt\nonumber\\
&=\int_x^1 e^{-2i \mu s }\sigma_1(s)\Big(\int_s^1e^{2i \mu t } \sigma_2(t)dt\Big)^2 ds\nonumber\\
&+\int_0^x \sigma_1(s)\int_s^1 e^{2i \mu(\tau-s )}\sigma_2(\tau)d\tau ds\int_x^1 e^{2i \mu t }\sigma_2(t)dt\nonumber\\
&=\int_x^1 e^{2i \mu s }\sigma_1(s)\Big(\int_s^1e^{2i \mu (t-s) } \sigma_2(t)dt\Big)^2 ds \nonumber\\
&+e^{2i\mu x}(K_1e)(x)\int_x^1 e^{2i \mu (t-x) }\sigma_2(t)dt\label{calkK1}.
\end{align}
Using \eqref{Ae} along with \eqref{Dgamma0} and \eqref{gammaTilde}, we arrive at \eqref{intK1}.

Now
applying a similar argument
 for the integral containing $K_2e,$ we obtain
\begin{align}
\int_0^x &e^{2i \mu (x-t) }\sigma_1(t)(K_2e)(t)dt\nonumber\\
&=\int_0^x e^{2i \mu (x-t) }\sigma_1(t)\int_t^1\sigma_2(s)\int_0^s e^{2i \mu (s-\tau) }\sigma_1(\tau)d \tau ds dt\nonumber\\
&=\int_0^x e^{2i \mu (s+x) }\sigma_2(s)\Big(\int_0^se^{-2i \mu t } \sigma_1(t)dt\Big)^2 ds\nonumber\\
&+\int_x^1 \sigma_2(s)\int_0^s e^{2i \mu(s- \tau )}\sigma_1(\tau)d\tau ds\int_0^x e^{2i \mu(x- t) }\sigma_1(t)dt\nonumber\\
&=\int_0^x e^{2i \mu(x- s) }\sigma_2(s)\Big(\int_0^se^{2i \mu(s-t) } \sigma_1(t)dt\Big)^2 ds\nonumber\\
&+(K_2e)(x)\int_0^x e^{2i \mu(x-t) }\sigma_1(t)dt\label{calkK2}
\end{align}
and second identity in \eqref{intK1} follows.
The two remaining identities are obvious in view of  the representations
\begin{align*}
\int_0^1 e^{2i \mu t }\sigma_2(t)(K_1e)(t)dt&=
\int_0^1 e^{2i \mu s }\sigma_1(s)\Big(\int_s^1e^{2i \mu (t-s) } \sigma_2(t)dt\Big)^2 ds\\
e^{2i \mu}\int_0^1 e^{-2i \mu t }\sigma_1(t)(K_2e)(t)dt&=
\int_0^1 e^{2i \mu (1-s) }\sigma_2(s)\Big(\int_0^se^{2i \mu (s-t) } \sigma_1(t)dt\Big)^2 ds.
\end{align*}
\end{proof}

Now we are able to obtain a crucial intermediate result providing
the required norm estimates for $K_j^2e$,
$j=1,2$.
\begin{cor}\label{corKsq}
If $\mu \in \Sigma_r$, $|\mu|\to \infty,$ then
\begin{align*}
  \|K_j^2e \|_\mathcal{C} & \leq   b \widetilde{\gamma}(\mu), \ \ j=1,2,
\end{align*}
where $b:=\| \sigma_0\|_{L^1}\Big(v(r)+\frac{1}{2}\Big)$.
\end{cor}
\begin{proof}
Let $j=1$. Using the definition of $K_1$ and  the representation \eqref{calkK1}, we write
\begin{align*}
  (K_1^2e)(x) & = \int_0^x \sigma_1(t) \int_t^1 e^{2i \mu (s-t) }\sigma_2(s)(K_1e)(s)dsdt\\
  &= \int_0^x \sigma_1(t)\int_t^1 e^{2i \mu (s-t) }\sigma_1(s)\Big(\int_s^1e^{2i \mu (\tau-s) } \sigma_2(\tau)d\tau\Big)^2 ds dt\\
  &+\int_0^x \sigma_1(t)\int_t^1 e^{2i \mu(y-t) }\sigma_2(y)dy\int_0^t \sigma_1(s)\int_s^1 e^{2i \mu(\tau-s )}\sigma_2(\tau)d\tau dsdt\\
  &=A(x)+B(x),
\end{align*}
and estimate each of the terms $A$ and $B$ separately.

To estimate $A$ observe that
\[|A(x)|\leq v(r)\|\sigma_1\|_{L^1}\int_0^1|\sigma_1(s)|\gamma_{0,2}^2(s,\mu)ds\leq v(r)\|\sigma_1\|_{L^1}\widetilde{\gamma}(\mu).\]
To obtain a bound for $B,$  note that $B$ can be rewritten as
\begin{align*}
B(x)=\int_0^x g'(t)g(t)dt \ \ \mbox{with} \ \ g(0)=0,
\end{align*}
where
\[g(t)=\int_0^t \sigma_1(s)\int_s^1 e^{2i \mu(\tau-s )}\sigma_2(\tau)d\tau ds.\]
From here it follows that
\[B(x)=\frac{1}{2}\Big(\int_0^x \sigma_1(s)\int_s^1 e^{2i \mu(\tau-s )}\sigma_2(\tau)d\tau ds\Big)^2,\]
and the Cauchy-Schwarz inequality yields
\[2|B(x)|\leq \|\sigma_1\|_{L^1}\int_0^1|\sigma_1(s)|\gamma_{0,2}^2(s,\mu)ds\leq \|\sigma_1\|_{L^1} \widetilde{\gamma}(\mu).\]

In the case $j=2,$ the proof goes along the same lines. Indeed, using \eqref{calkK2}, we infer that
\begin{align*}
  (K_2^2e)&(x)  = \int_x^1 e^{2i\mu t}\sigma_2(t)  \int_0^t e^{-2i\mu s }\sigma_1(s)(K_2e)(s)dsdt\\
   &= \int_x^1 \sigma_2(t)\int_0^t e^{2i \mu (t-s) }\sigma_2(s)\Big(\int_0^se^{2i \mu (s-\tau) } \sigma_1(\tau)d\tau\Big)^2 ds dt\\
  &+\int_x^1 \sigma_2(t)\int_0^t  e^{2i \mu(t-y)}\sigma_1(y)dy\int_t^1 \sigma_2(s)\int_0^s  e^{2i \mu(s-\tau )}\sigma_1(\tau)d\tau dsdt,
\end{align*}
so that one can apply the same arguments as in the case $j=1$.
\end{proof}

We can now easily prove the refined version of the Theorem \ref{AsG1}.
\begin{cor}\label{AsGC}
Under the assumptions of Theorem \ref{AsG1}, its conclusion remains valid, with the bound \eqref{Ase} replaced by
the following estimate
\begin{align*}
\|z_j-e+K_je\|_{\mathcal C}
\le  c \widetilde{\gamma}(\mu), \quad \mu \in \Sigma_r(R_r),
\end{align*}
for some $c>0.$
\end{cor}
\begin{proof}

Using Corollary \ref{corKsq} and Lemma \ref{K1M} it is
easy to see that for all $n \in \mathbb N,$
\begin{align}\label{K1200}
\|K_j^{2n}e\|_{\mathcal C}\le  b a^{n-1}\gamma_{j,q}^{n-1}(\mu) \widetilde{\gamma}(\mu), \ \ j=1,2.
\end{align}
Here $b$ is the constant from Corollary \ref{corKsq}.
Using this inequality and again representation of $z_j$ as a series, we get
\begin{align*}
\|z_j-e+K_je\|_{\mathcal C}
&\le
(1+\|K_j\|)\sum_{m=1}^\infty  \|K_j^{2m}e\|_{\mathcal C}
\\
& \leq b(1+v(r)\|\sigma_0\|^2_{L^p}) \widetilde{\gamma}(\mu)
\sum_{m=1}^\infty  a^{m-1} \gamma_{j,q}^{m-1}(\mu)\\
& \leq 2 b(1+v(r)\|\sigma_0\|^2_{L^p}) \widetilde{\gamma}(\mu),
\end{align*}
thus the proof is completed.
\end{proof}

\section{Fundamental system of solutions}

We will now focus on describing the asymptotic behavior of the fundamental system
of solutions $W=[w_1, w_2]^T$ and $V=[v_1,v_2]^T$ to
the Dirac system
\begin{align}\label{DI}
Y'(x)+J(x) Y(x)&=i\mu AY(x),\quad
 x\in [0,1],
\end{align}
satisfying conditions \eqref{BC}.

The main theorems of this section include three asymptotic formulas with three types of remainders.
The most accurate formulas will be stated first and so on.
Remainders will be estimated by the following expressions
\begin{align}\label{la}
\Xi_q(x,\mu) & :=\gamma_{1,q}(\mu)+\gamma_{2,q}(\mu)+
  \gamma_{0,1}(x,\mu)+\gamma_{0,2}(x,\mu),\\
  \Lambda_q(x,\mu) & :=\widetilde{\gamma}(\mu)+\gamma^2_{1,q}(\mu)+\gamma^2_{2,q}(\mu)+
  \gamma^2_{0,1}(x,\mu)+\gamma^2_{0,2}(x,\mu).
\end{align}

\begin{rem}
In the case $p=1$ and $q=\infty$  in all results  the expressions $\Xi_q(x,\mu)$ and $\Lambda_q(x,\mu)$ could be replaced
by
\begin{align}\label{inf}
\Xi_\infty(\mu)&:= \gamma_{1,\infty}(\mu)+\gamma_{2,\infty}(\mu),
\\
\Lambda_\infty(\mu)&:=\gamma^2_{1,\infty}(\mu)+\gamma^2_{2,\infty}(\mu).
\end{align}
\end{rem}

\begin{thm}\label{thmw}
Assume $r\ge 0$ is fixed and  take $R_r>0$  so large that
\begin{align*}
\gamma_{1,q}(\mu)\le 1/(2a), \quad \mbox{for}\quad \mu\in \Sigma_r(R_r).
\end{align*}
Then, for  $\mu\in \Sigma_r(R_r),$ there exists
the solution $W=[w_1, w_2]^T$ of \eqref{DI} satisfying $w_1(0)=1$, $w_2(1)=0$
which is analytic in $\mathcal{C}[0,1]\times \mathcal{C}[0,1]$ with respect to $\mu$ and
 admits the following
asymptotic representations as $|\mu|\to \infty:$
\begin{align}
w_1(x)&=e^{i \mu x}\Big(1-(K_1e)(x)\Big)+ e^{i \mu x}R_1(x,\mu), \label{A1}\\
w_2(x)&=e^{i \mu x}\Big(\int_x^1 e^{2i \mu (t-x)}\sigma_2(t)dt-\int_x^1e^{2i \mu (t-x)}\sigma_2(t)(K_1e)(t)dt \Big)\nonumber\\
&+ e^{i \mu x}R_2(x,\mu),\nonumber
\end{align}
where
$\|R_j(x,\mu)\|_C \leq C_j \widetilde{\gamma}(\mu)$ for a constant $C_j>0, j=1,2$.
Moreover, we have
\begin{align}
w_1(x)&=e^{i \mu x}\Big(1-(K_1e)(x)\Big)+ e^{i \mu x}S_1(x,\mu),\label{A2}\\
w_2(x)&=e^{i \mu x}\int_x^1 e^{2i \mu (t-x)}\sigma_2(t)dt
+ e^{i \mu x}S_2(x,\mu)\nonumber
\end{align}
with
$|S_j(x,\mu)| \leq C_j \Lambda_q(x,\mu)$,  $x\in [0,1]$, $j=1,2,$
and
\begin{align}
w_1(x)&=e^{i \mu x}+ e^{i \mu x}T_1(x,\mu),\label{A3}\\
w_2(x)&=e^{i \mu x}T_2(x,\mu),\nonumber
\end{align}
where
$|T_j(x,\mu)| \leq C_j \Xi_q(x,\mu),$  $x\in [0,1]$, $j=1,2$.
\end{thm}

\begin{proof}
Recall that for all $x \in [0,1],$ the functions $W=[w_1, w_2]^T$ and $z_1$ (the solution of \eqref{K1E}) are related by the transformation
\begin{align*}
w_1(x)&=e^{i \mu x}z_1(x),\\
w_2(x)&=e^{-i \mu x}\int_x^1 e^{2i \mu t}\sigma_2(t)z_1(t)dt.
\end{align*}
In Corollary \ref{AsGC}, we showed that
\begin{align}
z_1(x)=1-(K_1e)(x)+P_1(x,\mu), \qquad x \in [0,1],
\end{align}
where $\|P_1(x,\mu)\|_\mathcal{C}\leq c \widetilde{\gamma}(\mu)$
for some $c>0$.
This directly implies  \eqref{A1}.

Furthermore, applying  \eqref{calkK1} to \eqref{A1} we obtain
\begin{align}\label{A11}
w_2(x)&=e^{i \mu x}\Big(\int_x^1 e^{2i \mu (t-x)}\sigma_2(t)dt-(K_1e)(x)\int_x^1e^{2i \mu (t-x)}\sigma_2(t)dt \Big)\\
&+ e^{i \mu x}R_1(x,\mu),\notag
\end{align}
where $\|R_1(x,\mu)\|_{\mathcal C} \leq C_1 \widetilde{\gamma}(\mu)$ for some $C_1>0$.
Since by \eqref{Ae} we have
$\|K_1e\|_{\mathcal C} \leq c \gamma_{2,q}(\mu)$, the proof of \eqref{A2} is completed,
and \eqref{A3} also follows from
\eqref{Ae}.

Finally, the claim on analyticity of $W$ with respect to $\mu$ follows
from the fact that $z_1$ is analytic.
\end{proof}

Similarly, employing the transformations
\begin{align*}
v_1(x)&=-e^{i \mu (x+1)}\int_0^x e^{-2 i \mu t}\sigma_1(t)z_2(t)dt,\\
v_2(x)&= e^{i \mu (1-x)}z_2(x),
\end{align*}
and the representation
\begin{align}
z_2(x)=1-(K_2e)(x)+P_2(x,\mu),
\end{align}
where $\|P_2(x,\mu)\|_\mathcal{C}\leq c \widetilde{\gamma}(\mu)$,
one can prove a result for $V$ analogous to the one for $W$ given in Theorem \ref{thmw}.
\begin{thm}\label{thmv}
Let $r\ge 0$  be fixed and take $R_r>0$ so large that
\begin{align*}
\gamma_{2,q}(\mu)\le 1/(2a),\quad \mbox{for}\quad \mu\in \Sigma_r(R_r).
\end{align*}
Then for $\mu\in \Sigma_r(R_r)$ there exists
the solution $V=[v_1,v_2]^T$ of \eqref{DI} satisfying $v_1(0)=0$, $v_2(1)=1$
which is analytic in $\mathcal{C}[0,1]\times \mathcal{C}[0,1]$  with respect to $\mu$ and
 admits the
asymptotic representations as $|\mu|\to \infty:$
\begin{align}
v_1(x)&=e^{i \mu (1-x)}\Big(-\int_0^x e^{2i \mu (x-t)}\sigma_1(t)dt\label{B1}\\
&+\int_0^xe^{2i \mu (x-t)}\sigma_1(t)(K_2e)(t)dt \Big)+ e^{i \mu (1-x)}M_1(x,\mu),\nonumber\\
v_2(x)&=e^{i \mu(1-x)}\Big(1-(K_2e)(x)\Big)+ e^{i \mu (1-x)}M_2(x,\mu),\nonumber
\end{align}
where
$\|M_j(x,\mu)\|_C \leq C_j \widetilde{\gamma}(\mu)$ for some $C_j>0, j=1,2.$
Moreover, the following representations of the solutions hold:
\begin{align}
v_1(x)&=-e^{i \mu (1-x)}\int_0^x e^{2i \mu (x-t)}\sigma_1(t)dt + e^{i \mu (1-x)}N_1(x,\mu) \label{B2}\\
v_2(x)&=e^{i \mu(1-x)}\Big(1-(K_2e)(x)\Big)+ e^{i \mu (1-x)}N_2(x,\mu), \nonumber
\end{align}
with
$|N_j(x,\mu)| \leq C_j \Lambda_q(x,\mu)$,  $x\in [0,1]$, $j=1,2$
and
\begin{align}
v_1(x)&=e^{i \mu (1-x)}L_1(x,\mu)\label{B3}\\
v_2(x)&=e^{i \mu(1-x)}+ e^{i \mu (1-x)}L_2(x,\mu), \nonumber
\end{align}
where
$|L_j(x,\mu)| \leq C_j \Xi_q(x,\mu)$, $x\in [0,1]$, $j=1,2$.
\end{thm}
Observe that the solutions $W$ and $V$ described in the two preceding theorems constitute a fundamental system of solutions to the
Dirac system \eqref{DI}. Directly from the conditions \eqref{BC} it follows that they are linearly independent, so that $[W,V]$
is a fundamental matrix.

\begin{rem}\label{remNeg}
The results of Theorems \ref{thmw} and \ref{thmv} are true in particular for $\Im\mu > 0 $
and outside some ball centered at zero. Let us recall (see Remark 5.5 in \cite{Rzep25}) how
to obtain formulas valid for $\Im \mu < 0$.

We take
\[
\tilde{\mu}=-\mu, \quad \tilde{\sigma}_1(x)=\sigma_2(x),\quad \tilde{\sigma}_2(x)=\sigma_1(x),
\]
and
\[
\tilde{y}_1(x)=y_2(x),\quad \tilde{y}_2(x)=y_1(x).
\]
Then \eqref{sys} can be rewritten as
\begin{align}\label{DIn}
\tilde{y}_1'(x)&+\tilde{\sigma}_1(x)\tilde{y}_2(x)=i\tilde{\mu}\tilde{y}_1(x),\\
\tilde{y}_2'(x)&+\tilde{\sigma}_2(x)\tilde{y}_1(x)=-i\tilde{\mu} \tilde{y}_2(x).\nonumber
\end{align}
This is nothing else than \eqref{sys} and $\Im \tilde{\mu} >0$, thus we may apply the same theorems with
appropriate substitutions. We have to remember here that $W$ and $V$ satisfy conditions
\eqref{BC}, thus
carrying out transformations for $\tilde{\mu}=-\mu$ we will get a fundamental
set of solutions with different conditions.

If $W=[w_1,w_2]$ and $V=[v_1,v_2]$ are the solutions of \eqref{sys} with $\Im \mu <0$
obtained from the solutions $\widetilde{W}$ and $\widetilde{V}$ of \eqref{DIn} for $\tilde{\mu}=-\mu$, then
 $w_1(1)=0$, $w_2(0)=1$ and $v_1(1)=1$, $v_2(0)=0$.
\end{rem}

\begin{rem}
Let $Y=[y_1,y_2]^T$ be a solution of the system \eqref{sys}.
We adopt the following notation
\[
\hat{y}_1(x)=y_2(\mathcal{T}x),\quad \hat{y}_2(x)=y_1(\mathcal{T}x),
\]
and
\[
\hat{\sigma}_1(x)=\sigma_2(\mathcal{T}x),\quad \hat{\sigma}_2(x)=\sigma_1(\mathcal{T}x),
\]
where $\mathcal{T}$ is given by \eqref{simT}.
It follows that
\[
y_1(x)=\hat{y}_2(\mathcal{T}x),\quad y_2(x)=\hat{y}_1(\mathcal{T}x),\quad x\in [0,1].
\]
and going back to \eqref{sys} we obtain the system
\begin{align}\label{SSs}
\hat{y}_1'(x)&-\hat{\sigma}_1(x)\hat{y}_2(x)=i\mu\hat{y}_1(x),\\
\hat{y}_2'(x)&-\hat{\sigma}_2(x)\hat{y}_1(x)=-i\mu\hat{y}_2(x).\nonumber
\end{align}
Moreover, the following identities hold
\begin{align*}
 \hat{y}_1(0)&=y_2(1),\quad \hat{y}_1(1)=y_2(0),\\
\hat{y}_2(0)&=y_1(1),\quad \hat{y}_2(1)=y_1(0).
\end{align*}

Above considerations show that if for large $\mu\in \Sigma_r$ vector functions
$W=[w_1,w_2]^T$ and
$V=[v_1,v_2]^T$ are the solutions of \eqref{DI}, \eqref{BC},
and  $\hat{W}=[\hat{w}_1,\hat{w}_2]^T$ and
$\hat{V}=[\hat{v}_1,\hat{v}_2]^T$,
are the solutions of the system \eqref{SSs} with the same boundary conditions \eqref{BC}, then
\begin{align}\label{Symm}
w_1(x)&=\hat{v}_2(\mathcal{T}x),\quad w_2(x)=\hat{v}_1(\mathcal{T}x),
\\
v_1(x)&=\hat{w}_2(\mathcal{T}x),\quad v_2(x)=\hat{w}_1(\mathcal{T}x).\nonumber
\end{align}
It is clear that properties \eqref{Symm} are connected to
relationship between integral operators $K_1$ and $K_2$ given by \eqref{Con}.

The principal conclusion of these deliberations can be formulated as follows. Once
the asymptotic formulas for $W$ from Theorem  \ref{thmw} are obtained, the formulas for $V$
from Theorem \ref{thmv} can be derived.
It suffices to take into account the third and fourth relations from  \eqref{Symm}
and the fact $\hat{V}$ is the solution of the system \eqref{SSs}, hence in all formulas for
$W$ instead of the matrix $J(x)$ in \eqref{v120} (and consequently in \eqref{sys}) we deals with a matrix $-J^T(\mathcal{T}x)$.
\end{rem}

We want to compare our formulas with a recent result from  \cite{SaKos21} and \cite{SaKos}. The authors considered there a more general system
\begin{equation}\label{SS}
y'-By=\lambda Ay,\quad Y=[y_1,y_2]^T,
\end{equation}
with
\[
A=\left[\begin{array}{cc}
a_1(x) & 0\\
0 &a_2(x)
\end{array}
\right],\quad
B(x)=\left[\begin{array}{cc}
b_{11}(x) & b_{12}(x)\\
b_{21}(x) & b_{22}(x)
\end{array}
\right],
\]
where $\lambda\in \mathbb{C}$ is a large parameter.
Here all functions are assumed to be integrable on $[0,1]$ and $a_1(x)-a_2(x)>0$, a.e.
on $[0,1].$
In this case, it was proved in \cite{SaKos} in particular that
 if
\[
\lambda\in \Pi_{\kappa}=
\{\lambda:\,\Re\,\lambda\ge \kappa\}, \ \ \kappa\in \mathbb{R},
\]
then there exists the
fundamental matrix
$Y=Y(x,\lambda)$ for \eqref{SS} which can be represented  as
\begin{equation}\label{KSAs0}
Y=M(I+\Upsilon(\lambda))E,
\end{equation}
via diagonal matrices
\[
M(x)=\left[\begin{array}{cc}
m_1 & 0\\
0 & m_2
\end{array}
\right],\quad
m_j(x,\lambda)=\exp{\left(\int_0^x b_{jj}(t)\,dt\right)},\; \quad j=1,2,
\]
and
\[E(x,\lambda)=\left[\begin{array}{cc}
e_1 & 0\\
0 & e_2
\end{array}
\right], \ \
e_j(x,\lambda)=\exp{\left(\lambda\int_0^x a_{j}(t)\,dt\right)},\; j=1,2.
\]
Here
\begin{align}
\Upsilon(\lambda)=\max_{i,j,s,x}|v_{ij}(s,x,\lambda)|, \ \ i,j=1,2, \ \ s,x \in [0,1],
\end{align}
where
\begin{align*}
v_{11}(s,x,\mu)& = - \int_{\max(s,x)}^1 b_{12}(t)e^{-2\lambda(t-s)}dt,\\
v_{12}(s,x,\mu)& = - \int_{\max(s,x)}^1 b_{12}(t)e^{2\lambda(x-t)}dt,\\
v_{21}(s,x,\mu)& = - \int^{\min(s,x)}_0 b_{21}(t)e^{-2\lambda(x-t)}dt,\\
v_{22}(s,x,\mu)& = - \int^{\min(s,x)}_0 b_{21}(t)e^{2\lambda(t-s)}dt.
\end{align*}

Note that
\begin{align}
2\Upsilon(\lambda)&\ge 2\max_{i=1,2,\,x\in [0,1]}|v_{ii}(x,x,\lambda)|
\label{comp}\\
&\ge \max_{x\in [0,1]}|v_{11}(x,x,\lambda)|+\max_{x\in [0,1]}|v_{22}(x,x,\lambda)|
\nonumber\\
&=\left\|\int_x^1 b_{12}(t)e^{-2\lambda(t-x)}dt\right\|_{\mathcal{C}}
+\left\|\int_0^x b_{21}(t)e^{-2\lambda(x-t)}dt \right\|_{\mathcal{C}}\nonumber.
\end{align}
If the following conditions are satisfied
\begin{equation}\label{warPor}
-B=J \; \mbox{and} \;  a_1(x)=1, \ \ a_2(x)=-1, \ \ x\in [0,1],
\end{equation}
then \eqref{SS}
is precisely \eqref{v120} with $\lambda=i\mu$ and the last line in \eqref{comp} coincides with
$\Xi_\infty$ given by \eqref{inf} upon interchanging functions $\sigma_1$ and $\sigma_2$.

Moreover, if \eqref{warPor} holds, then from \cite[Thm. 1]{SaKos} (see also \cite[Cor. 2]{SaKos21}) follows that when $\sigma_j$, $j=1,2$, are absolutely continuous on $[0,1],$
and $\lambda\in \Pi_\kappa,$ then as $\lambda\to \infty:$
\begin{equation}\label{SKAs}
Y=\left[\begin{array}{cc}
y_{11} & y_{12}\\
y_{21} & y_{22}
\end{array}\right],
\end{equation}
\begin{align*}
\left[\begin{array}{c}
y_{11} \\
y_{21}
\end{array}\right]&=e^{\lambda x}\left[\begin{array}{c}
1-\frac{1}{2\lambda}\int_x^1\sigma_1(t)\sigma_2(t)\,dt+\mbox{o}(\lambda^{-1})
\\
-\frac{\sigma_2(x)}{2\lambda}+\mbox{o}(\lambda^{-1})
\end{array}\right],
\\
\left[\begin{array}{c}
y_{21} \\
y_{22}
\end{array}\right]&=e^{-\lambda x}\left[\begin{array}{c}
\frac{\sigma_1(x)}{2\lambda}+\mbox{o}(\lambda^{-1})\\
1-\frac{1}{2\lambda}\int_0^x \sigma_1(t)\sigma_2(t)\,dt+\mbox{o}(\lambda^{-1})
\end{array}\right].
\end{align*}
Note that, by construction in \cite{SaKos21} and \cite{SaKos}, we have
\begin{align}\label{warPoczS}
 y_{11}(1)=e^\lambda,\quad y_{21}(0)=0,
\quad y_{12}(1)=0,\quad y_{22}(0)=1.
\end{align}

On the other hand, if $\sigma_j\in L^1[0,1]$, $j=1,2,$ then $p=1$ and $q=\infty$, thus
the asymptotic representations of the fundamental matrix
$[W,V]$ with \eqref{BC}
obtained in \eqref{A2} and \eqref{B2}
can be rewritten as
\begin{align}\label{Int1}
w_1(x)&=e^{i \mu x}\Big(1-(K_1e)(x)\Big)+ e^{i \mu x}S_1(x,\mu),\\
w_2(x)&=e^{i \mu x}\int_x^1 e^{2i \mu (t-x)}\sigma_2(t)dt   !1
+ e^{i \mu x}S_2(x,\mu)\nonumber
\end{align}
and
\begin{align}\label{Int2}
v_1(x)&=-e^{i \mu (1-x)}\int_0^x e^{2i \mu (x-t)}\sigma_1(t)dt + e^{i \mu (1-x)}N_1(x,\mu)\\
v_2(x)&=e^{i \mu(1-x)}\Big(1-(K_2e)(x)\Big)+ e^{i \mu (1-x)}N_2(x,\mu),
\end{align}
where for $j=1,2,$
\[
\|S_j(x,\mu)\|_{\mathcal C} \leq C_{1,j} \Lambda_\infty(\mu),\quad \|N_j(x,\mu)\|_C \leq C_{2,j} \Lambda_\infty(\mu),\quad\mu\in \Sigma_r(R),
\]
for some constants $C_{1,j}, C_{2,j}>0$, where $\Lambda_\infty$ is given by \eqref{inf}.
From this we can easily derive \eqref{KSAs0}, using Remark \ref{remNeg} and take $\lambda=-i\mu,$ for $\mu\in \Sigma_r$.

We would like to rewrite our formulas
\eqref{Int1}, \eqref{Int2} in a form similar to \eqref{SKAs}, in order to better evaluate the obtained results.
Integrating by parts,
and using inequality
\begin{equation}\label{Sob}
\|f\|_C\le \|f\|_{W_1^1}:=\|f\|_{L^1}+\|f'\|_{L^1},\quad f\in W_1^1=W_1^1[0,1],
\end{equation}
we derive for $\mu \in \Sigma_r$
\begin{equation}
\Lambda_\infty(\mu) \le
\frac{4v^2(r)}{|\mu|^2}\Big(\|\sigma_1\|^2_{W_1^1}+\|\sigma_1\|^2_{W_1^1}\Big).
\end{equation}
Next, twice integrating by the parts and using again \eqref{Sob}
for nonzero $\mu \in \Sigma_r$ we can deduce the following estimates
\begin{equation}
\left|(K_1e)(x)+
\frac{1}{2i\mu}
\int_0^x\sigma_1(t)
\sigma_2(t)dt\right|\le \frac{2v(r)}{|\mu|^2}\|\sigma_1\|_{W_1^1}\|\sigma_2\|_{W_1^1}.
\end{equation}
and
\begin{equation}
  \left|(K_2e)(x)+
\frac{1}{2i\mu}
\int_x^1\sigma_1(t)
\sigma_2(t)dt\right|\le \frac{2v(r)}{|\mu|^2}\|\sigma_1\|_{W_1^1}\|\sigma_2\|_{W_1^1}.
\end{equation}
Integrating by parts again
we easily arrive at
\begin{equation}
\left|\int_x^1 e^{2i \mu (t-x)}\sigma_2(t)dt-
(e^{2i \mu (1-x)}\sigma_2(1)
 - \sigma_2(x))\frac{1}{2i\mu}\right|
= \frac{o(1)}{|\mu|}, \quad |\mu|\to\infty.
\end{equation}
and
\begin{equation}
\left|\int_0^x e^{2i \mu (x-t)}\sigma_1(t)dt-(- \sigma_1(x)+
e^{2i \mu x}\sigma_1(0))\frac{1}{2i\mu}\right|
= \frac{o(1)}{|\mu|}, \quad |\mu|\to\infty.
\end{equation}
Here we used the fact that
\[
\left\|\int_0^x e^{2i \mu (x-t)}\sigma'_1(t)dt\right\|_C
\to 0, \quad \left\|\int_x^1 e^{2i \mu (t-x)}\sigma'_2(t)dt\right\|_C
\to 0,
\]
as
$\mu\in \Sigma_r$ and $|\mu|\to\infty$.

By applying the estimates obtained above to \eqref{Int1} and \eqref{Int2},
we conclude that for $\mu \in \Sigma_r(R)$ and sufficiently large $R>0$ there hold
\begin{align*}
 w_1(x)&=e^{i \mu x}\Big(1+
\frac{1}{2i\mu}
\int_0^x\sigma_1(t)
\sigma_2(t)dt +\frac{\mbox{o}(1)}{\mu}
\Big), \\
  w_2(x)&=e^{i \mu x}\left(
(e^{2i \mu (1-x)}\sigma_2(1)
 - \sigma_2(x))\frac{1}{2i\mu}+\frac{\mbox{o}(1)}{\mu}\right),
\end{align*}
and
\begin{align*}
 v_1(x)&=-e^{i \mu (1-x)}\left((- \sigma_1(x)+
e^{2i \mu x}\sigma_1(0))\frac{1}{2i\mu}+\frac{\mbox{o}(1)}{\mu}\right),  \\
v_2(x)&=e^{i \mu(1-x)}\Big(1+
\frac{1}{2i\mu}
\int_x^1\sigma_1(t)
\sigma_2(t)dt+\frac{\mbox{o}(1)}{\mu}
\Big).
\end{align*}

Then after straightforward transformations according  Remark \ref{remNeg}, we obtain
a fundamental system
\[
[\widetilde{W},\widetilde{V}]=\left[\begin{array}{cc}
\tilde{w}_1 & \tilde{v}_1\\
\tilde{w}_2 & \tilde{v}_2
\end{array}\right],
\]
which, for
$\mu\in -\Sigma_r$,
$|\mu|>R,$ with sufficiently large $R>0$ can be written as
\begin{align}\label{Int1A}
\left[\begin{array}{c}
\tilde{w}_1 \\
\tilde{w}_2
\end{array}\right]
&=e^{-i\mu x}\left[\begin{array}{c}
(-e^{-2i \mu (1-x)}\sigma_1(1)
 + \sigma_1(x))\frac{1}{2i\mu}+\frac{\mbox{o}(1)}{\mu}
\\
1-\frac{1}{2i\mu}
\int_0^x\sigma_1(t)
\sigma_2(t)dt +\frac{\mbox{o}(1)}{\mu}
\end{array}
\right],\\
\left[\begin{array}{c}
\tilde{v}_1 \\
\tilde{v}_2
\end{array}\right]
&=e^{-i \mu(1-x)}\left[
\begin{array}{c}
1-\frac{1}{2i\mu}
\int_x^1\sigma_1(t)
\sigma_2(t)dt+\frac{\mbox{o}(1)}{\mu}
\\
(-\sigma_2(x)+
e^{-2i \mu x}\sigma_2(0))\frac{1}{2i\mu}+\frac{\mbox{o}(1)}{\mu}
\end{array}\right].
\end{align}

So, if $\sigma_j\in W^{1,1}[0,1]$, $j=1,2,$ and $\lambda=i\mu$, then
the formulas obtained above differ from those describing \eqref{SKAs}.
Here the solution $[\tilde{w}_1,\tilde{w}_2]^T$ corresponds to
$[y_{12},y_{22}]^T$ and
$[\tilde{v}_1,\tilde{v}_2]^T$ corresponds to
solution $[y_{11},y_{21}]^T$.
In our formulas we have an additional
$-e^{-2i \mu (1-x)}\sigma_1(1)/(2i\mu)$
for $\tilde{w}_1$ and
$e^{-2i \mu x}\sigma_2(0)/(2i\mu)$
for $\tilde{v}_2$.
Observe that the absence of these terms in \eqref{SKAs} contradicts the boundary conditions
$y_{21}(0)=0$ and $y_{12}(1)=0$.
Therefore, the formulas provided would hold only in the case when
$\sigma_1(1)=0$, $\sigma_2(0)=0$, or $b_{12}(1)=0$, $b_{21}(0)=0$.

Furthermore, we provide a simple example demonstrating why these terms are necessary in the asymptotic formulas \eqref{SKAs}.
It is sufficient to consider a Dirac system
\begin{equation}\label{dd}
y_1'(x)+y_2(x)=\lambda y_1(x),\quad y_2'(x)=-\lambda y_2(x),
\end{equation}
with constant matrices
\begin{equation}
B=\left[ \begin{array}{cc}
0 & -1\\
0 &  0
\end{array}\right],\quad
A=\left[ \begin{array}{cc}
1 & 0\\
0 & -1
\end{array}\right],
\end{equation}
in \eqref{SS}.
A straightforward calculation shows that the fundamental set of solutions to \eqref{dd}
satisfying boundary conditions \eqref{warPoczS} is
\[\left[\begin{array}{c}
y_{11}
\\
y_{21}\end{array}\right]=
e^{\lambda x}\left[\begin{array}{c}
1\\
0\end{array}\right],
\]
\[
\left[\begin{array}{c}
y_{12}\\
y_{22}\end{array}\right]=
e^{-\lambda x}\left[\begin{array}{c}
\int_x^1 e^{2\lambda(x-t)}\,dt
\\
1\end{array}\right]=
e^{-\lambda x}\left[\begin{array}{c}
\frac{1}{2\lambda}(1-e^{-2\lambda(1-x)})
\\
1\end{array}\right],
\]
where $\lambda \in \mathbb{C}$. We can see that these expressions differ from those
in \eqref{SKAs} but agree with our results in \eqref{Int1A}.

\section{Cauchy problem for Dirac system}
In this section we analyze the solutions to matrix Cauchy problem
associated with the Dirac system \eqref{DI}:
\begin{align}\label{D0}
D'(x)+J(x) D(x)&= i \mu A D(x),\quad
 x\in [0,1],\\
 D(0)=I,
\end{align}
where $\sigma_j\in L^p[0,1]$, $j=1,2$, $1\leq p \leq 2$ and $\mu \in \Sigma_r$.
Here $D$ denotes a $2\times2$ matrix with entries from $W^{1,1}[0,1]$. Due to
\cite[Theorem 1.2.1]{Zettl} there exists a unique solution of this Cauchy problem.

The matrix $D$ is composed of the solutions to \eqref{DI} denoted as
$\mathbf{c}=[c_1,c_2]^T$ and $\mathbf{s}=[s_1,s_2]^T$
satisfying conditions
\begin{align}\label{cwarunki}
c_1(0)=1, \qquad c_2(0)=0,
\end{align}
and
\begin{align}\label{swarunki}
s_1(0)=0, \qquad s_2(0)=1.
\end{align}

Note that, as simple calculations reveal,
\begin{align}\label{reprCS}
\mathbf{c}=W-\frac{w_2(0)}{v_2(0)}V,\qquad
\mathbf{s}=\frac{1}{v_2(0)}V,
\end{align}
where $W=[w_1,w_2]^T$, $V=[v_1,v_2]^T$ is the fundamental set of solutions to \eqref{DI} described in Theorems \ref{thmw} and \ref{thmv}.
Observe that $v_2(0)\neq 0$. It cannot be equal to zero, since
we have uniqueness of solutions and $V$ satisfies \eqref{BC}.

These relations indicate how the statements  from the preceding section describing $W$ and $V$  should now be used to derive asymptotic formulas for
$\mathbf{c}$ and $\mathbf{s}$. As direct consequences Theorems \ref{thmw} and \ref{thmv} we obtain as $\mu \in \Sigma_r$ and $|\mu| \to \infty:$
\begin{align}
w_2(0)&=\int_0^1 e^{2i \mu t }\sigma_2(t)dt+{\rm O}(\widetilde{\gamma}(\mu)),\label{w20}\\
v_2(0)&= e^{i \mu }\Big(1-(K_2e)(0)+{\rm O}(\widetilde{\gamma}(\mu))\Big). \label{v20}
\end{align}
Finally, recall that $(K^2_2e)(0)={\rm O}(\widetilde{\gamma}(\mu))$ and therefore
\begin{align}\label{przezv20}
\frac{e^{i \mu }}{v_2(0)}&= 1+(K_2e)(0)+{\rm O}(\widetilde{\gamma}(\mu)), \ \ |\mu|\to \infty.
\end{align}

Now we are ready to prove the main result of this section.

\begin{thm}\label{main1}
Let $\mathbf{c}=[c_1,c_2]^T$ and $\mathbf{s}=[s_1,s_2]^T$ be the solutions of
\eqref{DI}  satisfying
\eqref{swarunki} and \eqref{cwarunki}.
If  $\mu \in \Sigma_r$ and $|\mu|\to \infty,$ then the following
formulas hold uniformly in $x\in [0,1]$
\begin{align*}
c_1(x)
&=e^{i\mu x}\Bigg(1+\int_0^xe^{-2i\mu s}\sigma_1(s)\int_0^s e^{2i\mu \xi}\sigma_2(\xi) d \xi ds\Bigg) + e^{-i\mu x}Q(x,\mu)\\
&+e^{-i\mu x}{\rm O}(\widetilde{\gamma}(\mu))\\
c_2(x)&= e^{-i\mu x}\Bigg(-\int_0^xe^{2i \mu t}\sigma_2(t) dt\\
&-\int_0^xe^{2i \mu t}\sigma_2(t) \int_0^te^{-2i \mu s}\sigma_1(s)ds\int_0^se^{2i \mu \tau}\sigma_2(\tau) d\tau ds dt\Bigg)\\
&+e^{-i\mu x }{\rm O}(\widetilde{\gamma}(\mu)),
\end{align*}
\begin{align*}
s_1(x)&= e^{i\mu x}\Bigg(-\int_0^xe^{-2i \mu t}\sigma_1(t) dt\\
&-\int_0^xe^{2i \mu s}\sigma_1(t) \int_s^xe^{-2i \mu t}\sigma_2(t)\int_0^se^{-2i \mu \tau}\sigma_1(\tau) d\tau dt ds\Bigg)\\
&+e^{-i\mu x}{\rm O}(\widetilde{\gamma}(\mu)),\\
s_2(x)&= e^{-i\mu x}\Bigg(1+\int_0^x\sigma_2(t)e^{2i\mu t}\int_0^t\sigma_1(s)e^{-2i\mu s} ds dt\Bigg)+e^{-i\mu x}{\rm O}(\widetilde{\gamma}(\mu)).
\end{align*}
where
\begin{align}\label{L}
Q(x,\mu) = \int_0^xe^{2i \mu (x-t)}\sigma_1(t)\int_0^t\sigma_2(s)\int_0^se^{2i\mu(s-\tau)}\sigma_1(\tau)\int_0^{\tau} e^{2i \mu y}\sigma_2(y) dy d\tau ds dt
\end{align}
and $\widetilde{\gamma}(\mu)$ is given by \eqref{gammaTilde}.
\end{thm}
\begin{proof}
We combine the results on $W$ and $V$ obtained above with formulas \eqref{reprCS}.
In particular, we will use \eqref{przezv20}.

Direct calculations lead to the following identities for $s_1$ and $s_2$:
\begin{align}\label{s1s2}
s_1(x)&= e^{i\mu x}\Bigg(\big(-1-(K_2e)(0)\big)\int_0^xe^{-2i \mu t}\sigma_1(t) dt
\\
&+\int_0^xe^{-2i \mu t}\sigma_1(t)(K_2e)(t) dt\Bigg)+e^{-i\mu x}{\rm O}(\widetilde{\gamma}(\mu)),\notag \\
s_2(x)&= e^{-i\mu x}\Big(1+(K_2e)(0)-(K_2e)(x)\Big)+e^{-i\mu x}{\rm O}(\widetilde{\gamma}(\mu)).\notag
\end{align}
Let us recall that the operators $K_j$ are double integrals. In order to obtain the results stated in the theorem for both
$\mathbf{c}$ and $\mathbf{s}$, it is now necessary to perform a series of algebraic transformations on the corresponding integral expressions. For the sake of clarity, most of these transformations have been placed in the appendix.

To transform \eqref{s1s2} to the desired form we obtain the identities \eqref{ls2} and \eqref{ls1}, whose
tedious proofs are  postponed to Appendix.
Then, employing \eqref{ls2} and \eqref{ls1}, it is  easy to recast \eqref{s1s2} as required.

Now we focus on $c_2$. Simple transformations leads to identity
\begin{align*}
c_2(x)&= e^{-i\mu x}\Bigg(-\int_0^xe^{2i \mu t}\sigma_2(t) dt-\int_x^1e^{2i \mu t}\sigma_2(t)(K_1e)(t)dt\\
&+\int_0^1e^{2i \mu t}\sigma_2(t)(K_1e)(t) dt\\
&+\Big((K_2e)(x)-(K_2e)(0)\Big)\int_0^1e^{2i \mu t}\sigma_2(t) dt\Bigg)\\
&+e^{-i\mu x }{\rm O}(\widetilde{\gamma}(\mu)).
\end{align*}
Comparing this identity to \eqref{lc2}, we obtain  the required formula for $c_2.$

Turning to $c_1,$ observe that
\begin{align}
c_1(x)
&=e^{i\mu x}\Big(1-(K_{1}e)(x)\Big)\nonumber\\
&+e^{-i\mu x}\Bigg(\int_0^1e^{2i \mu t}\sigma_2(t) dt\cdot\int_0^xe^{2i \mu (x-t)}\sigma_1(t) dt \nonumber\\
&+ (K_2e)(0) \cdot\int_0^1e^{2i \mu t}\sigma_2(t) dt  \cdot\int_0^xe^{2i \mu (x-t)}\sigma_1(t) dt \label{c1r}\\
&-  \int_0^1e^{2i \mu t}\sigma_2(t) dt  \cdot\int_0^xe^{2i \mu (x-t)}\sigma_1(t)(K_2e)(t) dt \nonumber\\
&- \int_0^xe^{2i \mu (x-t)}\sigma_1(t) dt\cdot \int_0^1e^{2i \mu t}\sigma_2(t)(K_1e)(t) dt\Bigg)\nonumber\\
&+e^{-i\mu x}{\rm O}(\widetilde{\gamma}(\mu)). \nonumber
\end{align}
The double integrals in this formula for $c_1$ can be transformed as in \eqref{lc1}.
Moreover, note that using \eqref{ls1}, after a change of variables, we may write
\begin{align*}
(K_2e)(0) &  \int_0^xe^{2i \mu (x-t)}\sigma_1(t) dt - \int_0^xe^{2i \mu (x-t)}\sigma_1(t)(K_2e)(t) dt \\
& = \int_0^xe^{2i \mu (x-t)}\sigma_1(t)(K_{1,1}e)(t)  dt,
\end{align*}
where $K_{1,1}$ is given by \eqref{K11}.
The sum in the third and forth lines of \eqref{c1r} takes the form
\begin{align*}
(K_2e)(0) & \int_0^1e^{2i \mu t}\sigma_2(t) dt  \cdot\int_0^xe^{2i \mu (x-t)}\sigma_1(t) dt \\
&- \int_0^1e^{2i \mu t}\sigma_2(t) dt  \cdot\int_0^xe^{2i \mu (x-t)}\sigma_1(t)(K_2e)(t) dt \\
& = \int_0^xe^{2i \mu (x-t)}\sigma_1(t)(K_{1,1}e)(t)  dt\cdot \int_0^1  e^{2i \mu \tau}\sigma_2(\tau) d\tau.
\end{align*}
Next, the last expression can be further transformed as
\begin{align*}
\int_0^x&e^{2i \mu (x-t)}\sigma_1(t)(K_{1,1}e)(t)  dt\cdot \int_0^1e^{2i \mu \tau}\sigma_2(\tau) d\tau  \\
& =  \int_0^xe^{2i \mu (x-t)}\sigma_1(t)\int_0^t\sigma_2(s)\int_0^se^{2i\mu(s-\tau)}\sigma_1(\tau)\int_0^1e^{2i \mu y}\sigma_2(y) dy d\tau ds dt \\
&= \int_0^xe^{2i \mu (x-t)}\sigma_1(t)dt \cdot \int_0^1e^{2i \mu t}\sigma_2(t)(K_1e)(t) dt-(K_{1}^2e)(x)\\
& + Q(x,\mu).
\end{align*}
Note now that the first term in the sum above equals to the terms in the last line  of
 \eqref{c1r} taken with the opposite sign. 
Since the third term in this sum (given by the quadruple integral)  is essentially
$Q(x,\mu),$ it suffices to recall that $\|K^2_1e\|_{\mathcal{C}}\leq c_r\widetilde{\gamma}(\mu)$
for $\mu \in \Sigma_r.$
\end{proof}

Let us compare Theorem \ref{main1} with the results from \cite{Rzep21} obtained for $\mu$ in the strip
$|\Im \mu| \leq d $ for some $d>0$.
To this aim, we need to introduce several types of remainders used there:
\begin{align}\label{al}
\alpha_0(x,\mu)&:=\sum_{j=1}^2\left(\Big|\int_0^xe^{-2i\mu t}\sigma_j(t)dt\Big|+\Big|\int_0^xe^{2i\mu t}\sigma_j(t)dt\Big|\right),
\end{align}
\begin{align}\label{ga0}
\alpha_q(\mu)&:=\sum_{j=1}^2\left(\Big\|\int_0^xe^{-2i\mu t}\sigma_j(t)dt\Big\|_{L_q}+\Big\|\int_0^xe^{2i\mu t}\sigma_j(t)dt\Big\|_{L_q}
\right),
\end{align}
and
\begin{equation}\label{alphat}
\widetilde{\alpha}(\mu):=\int_0^1 \sigma_0(s)\alpha_0^2(s,\mu)\,ds,
\end{equation}
where $\sigma_0$ is given by \eqref{sigma}.

Note that there is a slight difference between the reminders in \cite{Rzep21} and the
reminders from \eqref{Dgamma0} and \eqref{Dgamma} used in this work.
Naturally, it stems from the fact that the reminders in \eqref{al},\eqref{ga0} and \eqref{alphat} are considered in the strip
and then both expressions $e^{\pm i\mu x}$ are then bounded.
 It is also worth noting that  the reminders from \cite{Rzep21}
involve only integrals from zero to $x$, whereas in
the reminders employed in this work one also encounters integrals from $x$ to 1.


First recall that \cite[Lemma 2.2]{Rzep21} yields
as $|\Im \mu| \leq d$ and $\mu \to \infty,$
\begin{align*}
c_1(x)&=e^{i\mu x}\Big(1+\int_0^xe^{-2i \mu t}\tilde{\sigma}_1(x,t)dt\Big)+{\rm O}(\widetilde{\alpha}(\mu)),\\
c_2(x)&= -e^{-i\mu x}\int_0^xe^{2i \mu t}\sigma_2(t) dt-\int_0^xe^{-i \mu( x-2t)}(T_{\sigma_2}\tilde{\sigma}_1)(x,t)dt+{\rm O}(\widetilde{\alpha}(\mu)),
\end{align*}
and
\begin{align*}
s_1(x)&= -\int_0^xe^{i \mu (x-2t)}\sigma_1(t) dt-\int_0^xe^{i \mu(x- 2t)}(T_{\sigma_1}\tilde{\sigma}_2)(x,t)dt+{\rm O}(\widetilde{\alpha}(\mu)),\\
s_2(x)&= e^{-i\mu x}\Big(1+\int_0^xe^{2i \mu t}\tilde{\sigma}_2(x,t)dt\Big)+{\rm O}(\widetilde{\alpha}(\mu)),
\end{align*}
where
$$\tilde{\sigma}_1(x,t)=\int_0^{x-t}\sigma_1(\xi+t)\sigma_2(\xi) d \xi, \ \ \tilde{\sigma}_2(x,t)=\int_0^{x-t}\sigma_1(\xi)\sigma_2(\xi+t) d \xi,$$
and
\begin{align*}
\int_0^xe^{-2i\mu t}(T_{\sigma_1}\tilde{\sigma}_2)(x,t)\,dt&=\int_0^xe^{-2i\mu t} \int_0^{x-t} \sigma_1(t+\xi)\int_0^t\sigma_2(\xi+\tau)\sigma_1(\tau) d\tau d\xi dt\\
\int_0^xe^{-2i\mu t}(T_{\sigma_2}\tilde{\sigma}_1)(x,t)\,dt&=\int_0^xe^{-2i\mu t} \int_0^{x-t} \sigma_2(t+\xi)\int_0^t\sigma_1(\xi+\tau)\sigma_2(\tau) d\tau d\xi dt.
\end{align*}
To compare this with the results in the present paper,
we transform the expressions above to a form convenient for comparison.

Changing variables, we obtain
\begin{align*}
\int_0^x e^{2i\mu t} \tilde{\sigma}_2(x,t)dt&=\int_0^xe^{2i\mu t}\int_0^{x-t}\sigma_1(\xi)\sigma_2(t+\xi) d \xi dt\nonumber\\
&= \int_0^x\sigma_2(t)e^{2i\mu t}\int_0^t\sigma_1(s)e^{-2i\mu s} ds dt
\end{align*}
and similarly
\begin{align*}
\int_0^x e^{-2ti\mu} \tilde{\sigma}_1(x,t)dt&=\int_0^xe^{-2i\mu t}\int_0^{x-t}\sigma_1(t+\xi)\sigma_2(\xi) d \xi dt\\
&=\int_0^xe^{-2i\mu s}\sigma_1(s)\int_0^s e^{2i\mu \xi}\sigma_2(\xi) d \xi ds,
\end{align*}
hence double integrals in the formulas for considered reminders agree.
Furthermore, recall that (see Propostion 4.5 \cite{Rzep21}.)
\begin{align*}
\int_0^xe^{-2i\mu t}& (T_{\sigma_1}\tilde{\sigma}_2)(x,t)\,dt\\
&= \int_0^x \int_0^{x-\tau}e^{2i\mu \xi}\sigma_2(\xi+\tau)\sigma_1(\tau)\int_{\tau+\xi}^xe^{-2i\mu z} \sigma_1(z) d z d\xi d\tau\\
&= \int_0^xe^{-2i\mu \tau} \sigma_1(\tau)\int_\tau^{x}e^{2i\mu y}\sigma_2(y)\int_{y}^xe^{-2i\mu z} \sigma_1(z) d z dy d\tau\\
&=\int_0^x\sigma_2(y)e^{2i\mu y} \Bigg(\int_0^{y}\sigma_1(\tau)e^{-2i\mu \tau} d\tau \int_y^x \sigma_1(z)e^{-2i\mu z}\, d z\Bigg) dy,
\end{align*}
where we first changed the order of integration and then  changed variables $z=t+\xi$ and $y=\tau+\xi$ and again the order.
Arguing in a similar way, we have
\begin{align*}
\int_0^xe^{2i\mu t}& (T_{\sigma_2}\tilde{\sigma}_1)(x,t)\,dt\\
&=\int_0^x\sigma_1(y)e^{-2i\mu y} \Bigg(\int_0^{y}\sigma_2(\tau)e^{2i\mu \tau} d\tau \int_y^x \sigma_2(z)e^{2i\mu z}\, d z\Bigg) dy\\
&=\int_0^x\sigma_2(z)e^{2i\mu z} \int_0^z \sigma_1(y)e^{-2i\mu y}\int_0^{y}\sigma_2(\tau)e^{2i\mu \tau} d\tau dy dz,
\end{align*}
thus triple integrals involved in the reminders are the same as well.
It remains to consider the quadruple integral $Q$ from the formula for $c_1$. In Proposition \ref{int4} (see details in Appendix) it was shown that
$$
|Q(x,\mu)|\le c_d\widetilde{\alpha}(\mu),
$$
for some $c>0,$ therefore it is a part of remainder used in  \cite[Lemma 2.2]{Rzep21}.

To summarize the above discussion, we have shown that the main parts of the formulas presented in the
Theorem \ref{main1} are the same as those in \cite[Lemma 2.2]{Rzep21}
and the remainder terms have equivalent forms for a spectral parameter within the strip.

Furthermore, the results from \cite[Theorem 1.1]{Rzep21} also follow from Theorem \ref{main1}.
The remainder in \cite[Theorem 1.1]{Rzep21} has the form equivalent in the horizontal strip to $\Lambda(x,\mu)$ given by \eqref{la}.
This can be easily obtained from the formulas for $\mathbf{c}$ and $\mathbf{s}$ used in the proof of Theorem \ref{main1} and estimates for $K_j, j=1,2$.

\section{Perturbed Dirac system}

Let us now consider a perturbed Dirac system for large $\mu \in \Sigma_r $ in the form
\begin{equation}\label{CPr1}
Y'(x)+J(x)Y(x)=i \mu A Y(x)+
\frac{P(x)}{\mu}Y(x),\quad x\in [0,1], \quad \mu \neq 0,
\end{equation}
where $J$ is given by \eqref{matrixA} and its entries belong to $L^p[0,1]$, $1\leq p \le 2$, whereas the entries of
$P=[p_{ij}]_{i,j=1}^2$ are from $L^1[0,1]$.
The case $p=2$ and $\mu$ with $|\Im\mu| \leq d$, $d>0$ was
considered in \cite{GRZ1}.

Let $D = [W,V]$ be a fundamental matrix of \eqref{v120}, where
$W$, $V$ is the fundamental set of solutions introduced in Theorems \ref{thmw} and \ref{thmv}.
Of course $D(x)=D(x,\mu)$ however, in what follows we shall suppress dependence on $\mu$ in the notation.
Note that by Liouville's theorem (see e.g. \cite[Theorem 1.2.2]{Zettl}),
$$
\det D(x)=\det D(0)=v_2(0), \qquad x\in [0,1].
$$
where we have also used the assumptions on  $W$ and $V$.
Therefore, for every $x\in [0,1],$
\begin{equation}\label{Dminus}
v_2(0)D^{-1}(x)=\left[\begin{array}{cc}
v_2(x) & -v_1(x)\\
-w_2(x) & w_1(x)
\end{array}\right].
\end{equation}

Using the classical variation of constants formula, which can be found e.g. in \cite[Theorem 1.3.1]{Zettl},
we infer that $Y$ is the solutions of
\eqref{CPr1} if and only if it satisfies the integral equation
\begin{align}\label{IntAm}
Y(x)=D(x)C+\frac{1}{\mu}D(x)\int_0^x D^{-1}(t) P(t) Y(t)\,dt,\quad x\in [0,1],
\end{align}
where $C=[a_1,a_2]^T$ is a constant.
Note that
\[
 v_2(0)D^{-1}(t) P(t) Y(t)=[Q_1(Y),Q_2(Y)]^T(t),
 \]
where
 \begin{align*}
Q_1(Y)& =
(v_2p_{11}-v_1p_{21}) y_1+(v_2p_{12}-v_1p_{22})y_2 \in L^1[0,1],
\\
Q_2(Y)&=(- w_2p_{11}+w_1p_{21})y_1+(-w_2p_{12}+w_1p_{22})y_2 \in L^1[0,1].
\end{align*}

Since a constant $C$ is arbitrary,  we may modify its second component
and, for $r \ge 0, R>0$ and $\mu \in \Sigma_r(R)$ we set
\[
a_1=a,\quad a_2=b-\frac{1}{v_2(0)}\int_0^1 Q_2(Y)(t)\,dt,
\]
where $a$ and $b$ are arbitrary constants.
Then, from \eqref{IntAm}, it follows that
\begin{equation}\label{IntAmN}
Y(x)=aW(x)+bV(x)
+\frac{1}{\mu v_2(0)}
D(x)\left[
\begin{array}{c}\int_0^x Q_1(Y)(t)\,dt\\
-\int_x^1 Q_2(Y)(t)\,dt
\end{array}\right].
\end{equation}

Let $r \geq 0$ be fixed.
Setting
$a=1$ and $b=0,$ we look for a solution $Y$ to \eqref{IntAmN} of the form
\[
Y(x)=e^{i\mu x}Z(x),\quad Z=\left[\begin{array}{c}
z_1\\ z_2
\end{array}\right].
\]
It follows that
\begin{equation}\label{OA1}
Z(x)= e^{-i\mu x}W(x)
+\frac{1}{\mu}
(\mathcal{A}_1(\mu)Z)(x),\qquad x\in [0,1],
\end{equation}
where $\mathcal{A}_1(\mu)$ is a bounded linear (integral) operator on the Banach
space $X=\mathcal C[0,1]\times \mathcal C[0,1]$
given by
\[
(\mathcal{A}_1(\mu) Z)(x)=
\frac{D(x)}{v_2(0)}e^{-i\mu x}
\left[\begin{array}{c}
\int_0^x e^{i\mu t} Q_1(Z)(t)\,dt\\
-\int_x^1 e^{i\mu t} Q_2(Z)(t)\,dt
\end{array}\right], \qquad x \in [0,1].
\]

Next, letting $a=0$ and $b=1,$ we look for solutions to \eqref{IntAmN} having the form
\[
Y(x)=e^{i\mu(1-x)}Z(x).
\]
Using \eqref{IntAmN}, we infer that
for all $\mu \in \Sigma_r(R),$
\begin{equation}\label{OA2}
Z(x)=e^{-i\mu(1-x)}V(x)+\frac{1}{\mu}
(\mathcal{A}_2(\mu)Z)(x), \qquad x\in [0,1],
\end{equation}
where $\mathcal{A}_2 (\mu)$ a bounded linear (integral) operator on
 $X=\mathcal C[0,1]\times \mathcal C[0,1]$
defined by
\[
(\mathcal{A}_2 (\mu)Z)(x)=
\frac{D(x)}{v_2(0)}e^{-i\mu (1-x)}
\left[\begin{array}{c}
\int_0^x e^{i\mu(1-t)} Q_1(Z)(t)\,dt\\
-\int_x^1 e^{i\mu(1-t)} Q_2(Z)(t)\,dt
\end{array}\right], \qquad x \in [0,1].
\]
Moreover,
applying the asymptotic formulas
\eqref{A3} and \eqref{B3}, together with \eqref{przezv20},
we infer that there exists $C=C(R)>0$ such that
\[
\|\mathcal{A}_j(\mu)\|_X\le C,\qquad \mu\in \Sigma_r(R), \quad j=1,2.
\]

Now,  estimating directly the tails
 of power series of $\mathcal A_j, j=1,2,$ we arrive at the next theorem.

\begin{thm}\label{T.7.2}
Let $r\geq 0$ be fixed. Then there exists $R>0$ and a fundamental system
\[
\widetilde{W}=\left[\begin{array}{c}
\tilde{w}_1\\
\tilde{w}_2\end{array}\right],\quad
\widetilde{V}=\left[\begin{array}{c}
\tilde{v}_1\\
\tilde{v}_2\end{array}\right],\quad
\]
of the perturbed system \eqref{CPr1} which is analytic by $\mu\in \Sigma_{r}(R)$ and as  $\mu\to\infty,$
the following formulas hold uniformly for $x\in [0,1]$
\[
\widetilde{W}=W+\frac{e^{i\mu x}}{\mu}\mathcal{A}_1(e^{-i\mu \cdot}W)+
e^{i\mu x}\mbox{O}(|\mu|^{-2}),
\]
\[
\widetilde{V}=V+\frac{e^{i\mu(1-x)}}{\mu}\mathcal{A}_2(e^{-i\mu(1-\cdot)}V)+
e^{i\mu(1-x)}\mbox{O}(|\mu|^{-2}),
\]
where $[W,V]$ is the fundamental matrix of solutions to \eqref{DI} described in Theorems \ref{thmw} and \ref{thmv}.
\end{thm}

Using again formulas
\eqref{A2} and \eqref{B2}, along with \eqref{przezv20}, we can deduce the following corollary.
\begin{cor}\label{C.7.3}
Let $r\geq 0$ be fixed. Then there exists $R>0$ and a fundamental system
\[
\widetilde{W}=\left[\begin{array}{c}
\tilde{w}_1\\
\tilde{w}_2\end{array}\right],\quad
\widetilde{V}=\left[\begin{array}{c}
\tilde{v}_1\\
\tilde{v}_2\end{array}\right],\quad
\]
of the perturbed system \eqref{CPr1} which is analytic by $\mu\in \Sigma_{r}(R)$ and as  $\mu\to\infty$ it admits the  representation
\begin{align*}
\widetilde{W}&=W+\frac{e^{i\mu x}}{\mu}\left[\begin{array}{c}
-\int_0^x e^{2i\mu(x-t)}\sigma_1(t)dt\int_x^1 e^{2i\mu(t-x)}p_{21}(t)dt\\
-\int_x^1 e^{2i\mu(t-x)}p_{21}(t)dt\end{array}\right]\\
&+\frac{e^{i\mu x}}{\mu}\left[\begin{array}{c}
\int_0^x p_{11}(t)dt\\
\int_x^1 e^{2i\mu(t-x)}\sigma_2(t)dt\int_0^x p_{11}(t)dt\end{array}\right]+
e^{i\mu x}\mbox{O}\big(\rho(\mu)\big),
\end{align*}
\begin{align*}
\widetilde{V}&=V+\frac{e^{i\mu(1-x)}}{\mu}\left[\begin{array}{c}
-\int_0^x e^{2i\mu(x-t)}p_{12}(t)dt\\
\int_x^1 e^{2i\mu(t-x)}\sigma_2(t)dt\int_0^x e^{2i\mu(x-t)}p_{12}(t)dt\end{array}\right]
\\&+\frac{e^{i\mu (1- x)}}{\mu}\left[\begin{array}{c}
\int_0^x e^{2i\mu(x-t)}\sigma_1(t)dt\int_x^1 p_{22}(t)dt\\
-\int_x^1 p_{22}(t)dt
\end{array}\right]+
e^{i\mu(1-x)}\mbox{O}\big(\rho(\mu)\big),
\end{align*}
where
\begin{equation}\label{rho}
  \rho_q(\mu)=|\mu|^{-1}k_P(\mu)+\gamma^2_{1,q}(\mu)+\gamma^2_{2,q}(\mu)+\widetilde{\gamma}(\mu)+|\mu|^{-2},
\end{equation}
$$k_P(\mu)= \int_0^1\|P(t)\|\Big(\gamma_{0,1}(t,\mu)+\gamma_{0,2}(t,\mu)\Big)dt,$$
where $\|P(t)\|=\sum_{k,j=1}^2 |p_{jk}(t)|$ and the asymptotics of $W$ and $V$ are given by \eqref{A1} and \eqref{B1}.
\end{cor}
This result corresponds to \cite[Theorem 3.2]{GRZ1}, where $\mu$ was supposed to be large
within horizontal strip in the complex plane.
\begin{rem}\label{remBC}
It is worth to underline that the fundamental set of solutions $\widetilde{W}$, $\widetilde{V}$ to
the perturbed system \eqref{CPr1} satisfies the same conditions \eqref{BC} as the solutions of
unperturbed system $W$, $V$. It means
\begin{equation}\label{BCt}
\tilde{w}_1(0)=1,\;\tilde{w}_2(1)=0,\quad \tilde{v}_1(0)=0,\;\tilde{v}_2(1)=1.
\end{equation}
This fact is a simple consequence of identities \eqref{OA1} and \eqref{OA2}.
\end{rem}
\begin{rem}\label{remPer}
For $\Im \mu <0$ we proceed analogously as in Remark \ref{remNeg}. Then to get representation of
the fundamental set in this half-plane one need to change places of the solution
to \eqref{CPr1} considered  with $-\mu$, $J^T$  and the matrix
$$\widetilde{P}=-\left[\begin{array}{cc} p_{22}&p_{21}\\ p_{12}& p_{11}\end{array}\right]$$
instead of $\mu$, $J$ and $P$ respectively.
\end{rem}

\section{Sturm--Liouville equation with singular potential}

In this section we consider a Sturm--Liouville equation
\begin{equation}\label{SE}
y''(x)+q(x)y(x)+\mu^2  y(x)=0,\quad x\in [0,1],\quad
\end{equation}
where the potential $q$ is a distribution
of the first order,  i.e.
\[
q=\sigma',\qquad \sigma\in L^2[0,1],
\]
and $\mu\in \mathbb{C}$ is a spectral parameter.
In the same way as in \cite{Atk} and \cite{SS} we follow the regularization method
to understand \eqref{SE} properly.
To this aim we introduce the quasi-derivative of $y\in W^{1,1}[0,1]$ by
\[
y^{[1]}(x):=y'(x)+\sigma(x)y(x),\qquad
y'(x)=y^{[1]}(x)-\sigma(x)y(x),
\]
and rewrite  \eqref{SE} as
\begin{equation}\label{SE10}
(y^{[1]}(x))'-\sigma(x)y^{[1]}(x)+\sigma^2(x)y(x)
+\mu^2 y(x)=0,\qquad x\in [0,1].
\end{equation}
We say that  $y$ is a solution of \eqref{SE10} if
\[
y\in \mathcal{D}:=\{y\in W^{1,1}[0,1]:
y^{[1]}(x)\in W^{1^,1}[0,1]\},
\]
and  \eqref{SE10} is satisfied  for a.e. $x\in [0,1]$.

Let us recall a very important transformation of \eqref{SE} from \cite[Section 4]{GRZ1}, which leads to
a perturbed Dirac system. An identical transformation appeared in a subsequent work \cite[section 6]{SSad3}
without any reference to the paper \cite{GRZ1}.
Note that \eqref{SE10} can be written in the next matrix form:
\begin{equation}\label{Znew}
L\left[\begin{array} {c}
y \\
y^{[1]}\end{array}
\right]:=
\frac{d}{dx}
\left[\begin{array} {c}
y \\
y^{[1]}\end{array}
\right]+
\left[\begin{array}{cc}
\sigma & -1\\
\sigma^2+\mu^2 & -\sigma\end{array}\right]\left[\begin{array} {c}
y \\
y^{[1]}\end{array}
\right]=0.
\end{equation}
To relate   \eqref{SE10}
to the perturbed system of the form \eqref{CPr1},
for $\mu\not=0,$ define
\[
S_1(\mu):=\left[\begin{array}{cc}
1 &  1\\
i\mu & -i\mu
\end{array}\right],\quad
S_1^{-1}(\mu)=
\frac{1}{2i\mu}\left[
\begin{array}{cc}
i\mu & 1\\
i\mu & -1
\end{array}\right]
\]
and functions $d_1,$  $d_2$ such that
\begin{equation}\label{Change1}
\left[\begin{array} {c}
y \\
y^{[1]}\end{array}
\right]=S_1(\mu)
\left[\begin{array} {c}
d_1 \\
d_2\end{array}
\right],\quad
\left[\begin{array} {c}
d_1 \\
d_2\end{array}
\right]=S_1^{-1}(\mu)
\left[\begin{array} {c}
y \\
y^{[1]}\end{array}
\right],
\end{equation}
then by \eqref{Znew},
\[
L\left[\begin{array} {c}
y \\
y^{[1]}\end{array}
\right]=S_1(\mu)L_0\left[\begin{array} {c}
d_1 \\
d_2\end{array}
\right],
\]
where
\[
L_0\left[\begin{array} {c}
d_1 \\
d_2\end{array}
\right]:=\left[\begin{array} {c}
d_1 \\
d_2\end{array}
\right]'+
\frac{1}{2i\mu}
\left[\begin{array}{cc}
2\mu^2+\sigma^2 & 2i\mu \sigma+\sigma^2\\
2i\mu \sigma-\sigma^2 & -2\mu^2-\sigma^2\end{array}\right]\left[\begin{array} {c}
d_1 \\
d_2\end{array}
\right].
\]
It was shown in \cite{GRZ1} that $y$ is a solution of
\eqref{SE10}
if and only if $D=[d_1,d_2]^T$, given by
\eqref{Change1}, is a solution of the perturbed Dirac system
\begin{equation}\label{CaychyInt0}
D'(x)+\sigma(x)\left[\begin{matrix}
0 & 1\\
1 & 0\end{matrix}\right]D(x)
=i\mu A D(x)+\frac{i\sigma^2(x)}{2\mu}\left[\begin{matrix}
1 & 1\\
-1 & -1\end{matrix}\right]D(x).
\end{equation}
The latter is precisely \eqref{CPr1}
with
\begin{equation*}
J(x)=\left[\begin{array}{cc}
0 &  \sigma(x)\\
\sigma(x) & 0\end{array}\right],\quad \sigma\in L^2[0,1],
\end{equation*}
and
\begin{equation}\label{rfun0}
P(x)=\frac{i}{2}\left[\begin{array}{cc}
\tau(x) &  \tau(x)\\
-\tau(x) & -\tau(x)\end{array}\right],\quad \tau=\sigma^2\in L^1[0,1].
\end{equation}

Due to \eqref{Change1} we have
\begin{equation}\label{change2}
\left[\begin{array} {c}
y \\
y^{[1]}\end{array}
\right]=
\left[\begin{array} {c}
d_1+d_2 \\
i\mu(d_1-d_2)\end{array}
\right].
\end{equation}
We apply this transformation to $\widetilde{W}=[\tilde{w}_1,\tilde{w}_2]^T$  the fundamental solution of \eqref{CPr1},
whose asymptotic behavior is described in Corollary \ref{C.7.3}.
Then we do the same for  $\widetilde{V}=[\tilde{v}_1,\tilde{v}_2]^T$ and then use  \eqref{A1} and \eqref{B1} to derive asymptotic formulas for
fundamental system of solutions to Sturm-Liouville problem with a singular potential.
\begin{thm}\label{Sl}
Let $r\geq 0$ be fixed. Then there exists $R>0$ and the fundamental set of solutions $[y_1,y_2]^T$ to \eqref{SE} which
is analytic in $\mu \in \Sigma_r(R)$ and admits the following representation for
 $|\mu| \to \infty:$
\begin{align*}
y_1(x)&=e^{i \mu x}\Bigg(1+\int_x^1 e^{2i \mu (t-x)}\sigma(t)dt-\int_0^x \sigma(t)\int_t^1e^{2i\mu(s-t)}\sigma(s)dsdt\nonumber\\
&- \int_x^1 e^{2i \mu (t-x) }\sigma(t)\int_0^t\sigma(s)\int_s^1 e^{2i \mu (\tau-s) }\sigma(\tau)d \tau ds dt\Bigg)\\
&+ \frac{ie^{i\mu x}}{2\mu}\Bigg(\int_0^x \sigma^2(t)dt+\int_0^x e^{2i\mu(x-t)}\sigma(t)dt\int_x^1 e^{2i\mu(t-x)}\sigma^2(t)dt\\
&+\int_x^1 e^{2i\mu(t-x)}\sigma^2(t)dt
+ \int_x^1 e^{2i\mu(t-x)}\sigma(t)dt\int_0^x\sigma^2(t)dt\Bigg)
+{\rm O}\big(\rho_2(\mu)\big)\\
y_2(x)&=e^{i \mu(1-x)}\Bigg(1-\int_0^x e^{2i \mu (x-t)}\sigma(t)dt-\int_x^1 \sigma(t)\int_0^te^{2i\mu(t-s)}\sigma(s)dsdt\nonumber\\
&+ \int_0^x e^{2i \mu (x-t) }\sigma(t)\int_t^1\sigma(s)\int_0^s e^{2i \mu (s-\tau) }\sigma(\tau)d \tau ds dt \Bigg)\\
&+\frac{ie^{i\mu(1-x)}}{2\mu}\Bigg(-\int_0^x e^{2i\mu(x-t)}\sigma^2(t)dt-\int_0^x e^{2i\mu(x-t)}\sigma(t)dt\int_x^1 \sigma^2(t)dt\nonumber\\
&
+\int_x^1 \sigma^2(t)dt+\int_x^1 e^{2i\mu(t-x)}\sigma(t)dt\int_0^x e^{2i\mu(x-t)}\sigma^2(t)dt\Bigg)+{\rm O}\big(\rho_2(\mu)\big),
\end{align*}
where $\rho_2$ is given by \eqref{rho}.
\end{thm}
\begin{rem}
Note that the formulas for quasi-derivatives  $y^{[1]}_j$, $j=1,2$ (which is also analytical by $\mu$), could be derived in the same way using
the second identity from \eqref{change2}. What is more, the fundamental set of solutions described in satisfies
the following conditions
\begin{align*}
i\mu y_1(0)+y_1^{[1]}(0)&=2i\mu, \ \ i\mu y_1(1)-y_1^{[1]}(1)=0\\
i\mu y_2(0)+y_2^{[1]}(0)&=0, \ \ i\mu y_2(1)-y_2^{[1]}(1)=2i\mu.
\end{align*}
This fact follows from Remark \ref{remBC} and transformation \eqref{change2}.
\end{rem}

\begin{rem}
According to Remark \ref{remPer} and the simpler representation of $J$ and $P$ in this case
for $\Im \mu <0$ we have to write the formulas for solutions of
\eqref{CaychyInt0} changing only $d_1$ and $\mu$ to $d_2$ and $-\mu$ respectively. The matrices $J$ and $P$ remains the same.
Due to \eqref{change2} for $y$ and $y^{[1]}$ this leads to identical formulas as for $\Im \mu >0$ but taking
$-\mu$ instead of $\mu$.
\end{rem}

\section{Appendix}
We include here several algebraic transformations for integral operators $K_j$, $j=1,2$ that will be useful in the
proof of Theorem \ref{main1}. We also provide an estimate of the expression
$Q$, which appears in the aforementioned theorem and is needed compare different results.

Note the following decompositions of $K_j,j=1,2$ hold:
\begin{align}\label{Sum1}
(K_1z)(x)=(K_{1,1}z)(x)+
(K_{1,2}z)(x),\qquad x\in [0,1],\quad z\in \mathcal{C}[0,1],
\end{align}
where
\begin{align}\label{K11}
(K_{1,1}z)(x)&:=
\int_0^x\sigma_2(s)
\left\{\int_0^s e^{2i\mu(s-t)}
\sigma_1(t)\,dt\right\}z(s)\,ds,\\
(K_{1,2}z)(x)&:=
\int_x^1e^{2i\mu s}\sigma_2(s)
z(s)\,ds\int_0^x e^{-2i\mu t}
\sigma_1(t)\,dt.\label{K12}
\end{align}
and
\begin{equation}
(K_2z)(x)=(K_{2,1}z)(x)+
(K_{2,2}z)(x),\quad x\in [0,1],\quad z\in \mathcal{C}[0,1],
\end{equation}
with
\begin{align}\label{K21}
(K_{2,1}z)(x)&:=
\int_x^1\sigma_1(s)
\left\{\int_s^1 e^{2i\mu(t-s)}
\sigma_2(t)\,dt\right\}z(s)\,ds,\\
(K_{2,2}z)(x)&:=
\int_0^xe^{-2i\mu s}\sigma_1(s)z(s)\,ds \int_x^1 e^{2i\mu t}
\sigma_2(t)\,dt \label{K22}.
\end{align}

The next lemma is a direct consequence of the formulas above.
\begin{lem}\label{Ksq}
The following identities hold:
\begin{align}
(K_2e)(0)-(K_2e)(x)=(K_{1,1}e)(x),\label{ls2}
\end{align}
\begin{align}
\int_0^xe^{-2i \mu t}\sigma_1(t) dt\int_0^1e^{2i \mu t}\sigma_2(t) dt&-(K_{1}e)(,\nonumber\\
& =\int_0^xe^{-2i\mu s}\sigma_1(s)\int_0^s e^{2i\mu \xi}\sigma_2(\xi) d \xi ds.\label{lc1}
\end{align}
\begin{align}\label{K1short}
(K_{1}e)(x)&=
\int_0^x\sigma_2(s)
\left\{\int_0^s e^{2i\mu(s-t)}
\sigma_1(t)\,dt\right\}\,ds+{\rm O}\big(\gamma_{0,1}(x)\gamma_{0,2}(x)\big),
\end{align}

\begin{align}\label{K2short}
(K_{2}e)(x)&=\int_x^1\sigma_1(s)
\left\{\int_s^1 e^{2i\mu(t-s)}
\sigma_2(t)\,dt\right\}\,ds
+{\rm O}\big(\gamma_{0,1}(x)\gamma_{0,2}(x)\big)
\end{align}
\begin{align*}
(K_{1,1}e)(x)+(K_{2,1}e)(x)&=(K_{1,1}e)(1)+{\rm O}\big(\gamma_{0,1}(x)\gamma_{0,2}(x)\big).
\end{align*}
\end{lem}

Next we derive another series of related identities.
\begin{lem}
We have
\begin{align}
-(K_2e)(0)&\int_0^xe^{-2i \mu t}\sigma_1(t) dt
+\int_0^xe^{-2i \mu t}\sigma_1(t)(K_2e)(t) dt\nonumber\\
&=-\int_0^xe^{2i \mu s}\sigma_2(s) \int_s^xe^{-2i \mu t}\sigma_1(t)\int_0^se^{-2i \mu \tau}\sigma_1(\tau) d\tau dt ds,\label{ls1}
\end{align}

\begin{align}
\Big((K_2e)(x)&-(K_2e)(0)\Big)\int_0^1e^{2i \mu t}\sigma_2(t) dt
+\int_0^1e^{2i \mu t}\sigma_2(t)(K_1e)(t) dt \nonumber\\
&-\int_x^1e^{2i \mu t}\sigma_2(t)(K_1e)(t) dt \nonumber\\
&=-\int_0^xe^{2i \mu t}\sigma_2(t) \int_0^te^{-2i \mu s}\sigma_1(s)ds\int_0^se^{2i \mu \tau}\sigma_2(\tau) d\tau ds dt.\label{lc2}
\end{align}
\end{lem}

\begin{proof}
First, note that using \eqref{ls2} and then changing the order of integration, we obtain
\begin{align*}
-(K_2e)(0)&\int_0^xe^{-2i \mu t}\sigma_1(t) dt
+\int_0^xe^{-2i \mu t}\sigma_1(t)(K_2e)(t) dt\\
&=
-\int_0^xe^{-2i \mu t}\sigma_1(t)(K_{1,1}e)(t) dt\\
&=-\int_0^xe^{-2i \mu t}\sigma_1(t) \int_0^te^{2i \mu s}\sigma_2(s)\int_0^se^{-2i \mu \tau}\sigma_1(\tau) d\tau ds dt\\
&=-\int_0^xe^{2i \mu s}\sigma_2(s) \int_s^xe^{-2i \mu t}\sigma_1(t)\int_0^se^{-2i \mu \tau}\sigma_1(\tau) d\tau dt ds.
\end{align*}

Considering next  the first two terms from the left-hand side of \eqref{lc2}, we infer that
\begin{align}
-(K_2e)(0)&\int_0^1e^{2i \mu t}\sigma_2(t) dt
+\int_0^1e^{2i \mu t}\sigma_2(t)(K_1e)(t) dt\notag \\
&=-\int_0^1e^{2i \mu t}\sigma_2(t) \int_0^te^{-2i \mu s}\sigma_1(s)\int_0^1e^{2i \mu \tau}\sigma_2(\tau) d\tau ds dt\label {k2}\\
&+\int_0^1e^{2i \mu t}\sigma_2(t) \int_0^te^{-2i \mu s}\sigma_1(s)\int_s^1e^{2i \mu \tau}\sigma_2(\tau) d\tau ds dt\notag \\
&=-\int_0^1e^{2i \mu t}\sigma_2(t) \int_0^te^{-2i \mu s}\sigma_1(s)\int_0^se^{2i \mu \tau}\sigma_2(\tau) d\tau ds dt,\notag
\end{align}
and the other two terms give
\begin{align}
(K_2e)(x)&\int_0^1e^{2i \mu t}\sigma_2(t) dt
-\int_x^1e^{2i \mu t}\sigma_2(t)(K_1e)(t) dt\notag \\
&=\int_x^1e^{2i \mu t}\sigma_2(t) \int_0^te^{-2i \mu s}\sigma_1(s)ds\int_0^se^{2i \mu \tau}\sigma_2(\tau) d\tau ds dt\label{k22} \\
&=\int_0^1e^{2i \mu t}\sigma_2(t) \int_0^te^{-2i \mu s}\sigma_1(s)ds\int_0^se^{2i \mu \tau}\sigma_2(\tau) d\tau ds dt\notag \\
&-\int_0^xe^{2i \mu t}\sigma_2(t) \int_0^te^{-2i \mu s}\sigma_1(s)ds\int_0^se^{2i \mu \tau}\sigma_2(\tau) d\tau ds dt.\notag
\end{align}
Therefore, adding \eqref{k2} and \eqref{k22} by sides, we conclude that
\begin{align*}
-(K_2e)(0)&\int_0^1e^{2i \mu t}\sigma_2(t) dt
+\int_0^1e^{2i \mu t}\sigma_2(t)(K_1e)(t) dt\\
&+(K_2e)(x)\int_0^1e^{2i \mu t}\sigma_2(t) dt
-\int_x^1e^{2i \mu t}\sigma_2(t)(K_1e)(t) dt\\
&=-\int_0^xe^{2i \mu t}\sigma_2(t) \int_0^te^{-2i \mu s}\sigma_1(s)ds\int_0^se^{2i \mu \tau}\sigma_2(\tau) d\tau ds dt
\end{align*}
and this completes the proof of second identity.
\end{proof}

We continue with providing a bound for $Q$ which appears in
the formula for the solution $c_1$ in Theorem \ref{main1}.
\begin{prop}\label{int4}
If $\mu \in \mathbb C$ satisfies $|\Im \mu|\le d$, $d>0$
then
\begin{equation}\label{A}
|Q(x,\mu)|\le 4 e^{4d}\|\sigma_1\|_{L^1}\widetilde{\alpha}(\mu),\qquad x \in [0,1].
\end{equation}
where $Q$ and $\tilde{\alpha}$ are given by \eqref{L} and \eqref{alphat} respectively.
\end{prop}

\begin{proof}
First, note that $Q$ can be rewritten as
\begin{equation}\label{LM}
Q(x,\mu)=\int_0^x e^{2i\mu(x-t)}\sigma_1(t)M(t,\mu)\,dt,
\end{equation}
where
\[
M(t,\mu)=\int_0^t e^{2i\mu s} \sigma_2(s)\int_0^s\sigma_1(\tau)e^{-2i\mu \tau}
\left(\int_0^\tau e^{2i\mu y}\sigma_2(y)\,dy\right)
 d\tau ds.
\]

Next, changing the order of integrals, we have
\begin{align*}
M(t,\mu)
&=\int_0^t e^{-2i\mu\tau} \sigma_1(\tau)
\left(\int_0^\tau e^{2i\mu y}\sigma_2(y)\,dy\right)
\int_\tau^t\sigma_2(s)e^{2i\mu s}
ds d\tau
\end{align*}
and using
\[
\int_\tau^t\sigma_2(s) e^{2i\mu s}
ds =\int_0^t\sigma_2(s)e^{2i\mu s}
ds -\int_0^\tau\sigma_2(s)e^{2i\mu s}\,ds,
\]
we obtain
\[
M(t,\mu)=M_1(t,\mu)-M_2(t,\mu),
\]
where
\[
M_1(t,\mu)=\int_0^t e^{-2i\mu\tau} \sigma_1(\tau)
\left(\int_0^\tau e^{2i\mu y}\sigma_2(y)\,dy\right)
\int_0^t\sigma_2(s)e^{2i\mu s}
ds d\tau,
\]
and
\[M_2(t,\mu)=\int_0^t e^{-2i\mu\tau} \sigma_1(\tau)
\left(\int_0^\tau e^{2i\mu y}\sigma_2(y)\,dy\right)^2
d\tau
\]
Therefore we have
\[
|M_2(t,\mu)|\le e^{2d}\widetilde{\alpha}(\mu), \ \ |\Im \mu| \leq d, \ \  t\in [0,1].
\]
Now we pass to deriving estimates for  $M_1:$
\begin{align*}
|M_1(t,\mu)|&\le e^{2d}\int_0^t |\sigma_1(\tau)|
\left|\int_0^\tau e^{2i\mu y}\sigma_2(y)\,dy\right|
\left|\int_0^t\sigma_2(s)e^{2i\mu s}
ds\right|\, d\tau\\
&\le \frac{e^{2d}}{2}\int_0^t |\sigma_1(\tau)|\left(
\left|\int_0^\tau e^{2i\mu y}\sigma_2(y)\,dy\right|^2+
\left|\int_0^t\sigma_2(s)e^{2i\mu s}
ds\right|^2\right)\, d\tau\\
&\le \frac{e^{2d}}{2}\left(\widetilde{\alpha}(\mu)+\|\sigma_1\|_{L^1}
\left|\int_0^t\sigma_2(s)e^{2i\mu s}
ds\right|^2\right).
\end{align*}
Combining the bounds for $M_1$ and $M_2,$
we infer that
\begin{align*}
2 e^{-2d}|M(t,\mu)|&\le |M_1(t,\mu)|+|M_2(t,\mu)|\\
&\le 3\widetilde{\alpha}(\mu)+
 \|\sigma_1\|_{L^1}
\left|\int_0^t\sigma_2(s)e^{2i\mu s}
ds\right|^2.
\end{align*}
Returning now to \eqref{LM}, we conclude that
\begin{align*}
2 e^{-4d}|Q(x,\mu)|&\le 2 e^{-2d}\int_0^1 |\sigma_1(t)||M(t,\mu)|\,dt
\\
&\le\int_0^1|\sigma_1(t)|
\left(3\tilde{\gamma}(\mu)+\|\sigma_1\|_{L^1}
\left|\int_0^t\sigma_2(s)e^{2i\mu s}
ds\right|^2\right)\,dt
\\
&\le 4\|\sigma_1\|_{L^1}\widetilde{\alpha}(\mu),
\end{align*}
which finishes the proof.
\end{proof}

\bibliographystyle{siam}
\bibliography{bibi2}

\end{document}